\documentclass[12pt,reqno]{amsart}

\usepackage[utf8]{inputenc}
\usepackage[T1]{fontenc}
\usepackage{enumerate}
\usepackage{hyperref}
\hypersetup{
  colorlinks   = true,
  citecolor    = blue,
  linkcolor    = blue
}

\usepackage{amsmath,amsfonts,amsthm,amssymb,amsrefs,color,tikz, comment}
\usepackage{mathrsfs}   
 \usepackage[normalem]{ulem}
\usepackage{pdfsync}
\usepackage[font={scriptsize}]{caption}

\usepackage[left=1in, right=1in, top=1.1in,bottom=1.1in]{geometry}
\usepackage[toc,page]{appendix}
  
\newcommand{\der}{\delta}

\newcommand{\E}{\mathbb E}
\newcommand{\R}{\mathbb R}
\newcommand{\N}{\mathbb N}
\newcommand{\bP}{\mathbb P}

\newcommand{\bex}{\mathbf{E}}

\newcommand{\bp}{\mathbf{P}}   
   
\newcommand{\cf}{\mathcal F}

\newcommand{\al}{\alpha}

\newcommand{\ep}{\varepsilon}

\newcommand{\ga}{\gamma}

\newcommand{\la}{\lambda}

\newcommand{\si}{\sigma}

\newcommand{\ups}{\upsilon}

\newcommand{\lp}{\left(}
\newcommand{\rp}{\right)}
\newcommand{\lc}{\left[}
\newcommand{\rc}{\right]}
\newcommand{\lcl}{\left\{}
\newcommand{\rcl}{\right\}}
\newcommand{\lln}{\left|}
\newcommand{\rrn}{\right|}

\newtheorem{theorem}{Theorem}[section]

\newtheorem{assumption}[theorem]{Assumption}

\newtheorem{corollary}[theorem]{Corollary}

\newtheorem{lemma}[theorem]{Lemma}

\newtheorem{proposition}[theorem]{Proposition}

\theoremstyle{remark}
\newtheorem{remark}[theorem]{Remark}

\theoremstyle{remark}

\newcommand{\bean}{\begin{eqnarray*}}
\newcommand{\eean}{\end{eqnarray*}}
\newcommand{\ben}{\begin{enumerate}}
\newcommand{\een}{\end{enumerate}}
\newcommand{\beq}{\begin{equation}}
\newcommand{\eeq}{\end{equation}}

\newcommand{\rsg}{~$\textnormal{RS}(G,p)/G/1$~}
\newcommand{\dgi}{~$\Delta_{(i)}/G/1$~}
\newcommand{\kmt}{~Koml\'os-Major-Tusn\'ady~}
\begin{document}

\title[Transitory Strong Embeddings]{Strong Embeddings for Transitory Queueing Models}

\address{Prakash Chakraborty: Department of Statistics,
  Purdue University,
  W. Lafayette, IN 47907,
  USA.}
\email{chakra15@purdue.edu}

\address{Harsha Honnappa: School of Industrial Engineering, 
  Purdue University, 
  W. Lafayette, IN 47906
  USA.}
\email{honnappa@purdue.edu}

\thanks{Authors supported by the National Science Foundation through grant CMMI/1636069.}


\author[P. Chakraborty, H. Honnappa]
{Prakash Chakraborty, Harsha Honnappa}

\begin{abstract}
In this paper we establish strong embedding theorems, in the sense of the\kmt framework, for the performance metrics of a general class of transitory models for nonstationary queueing systems. The non-stationary and non-Markovian nature of these models makes the computation of performance metrics hard. The strong embeddings yield error bounds on sample path approximations by diffusion processes, in the form of functional strong approximation theorems.
\end{abstract}

\maketitle

\section{Introduction}
In this paper we establish strong embedding theorems, in the sense of the\kmt (KMT) framework, for the performance metrics of a general class of {\em transitory queueing models}~\cite{HoJaWa2016a,BeHoLe2019}. 
Transitory queueing models assume a large, but finite population of customers arrive at the system over some time horizon. Examples of such systems include hospital surgery departments and clinics, subscription-based services such as video and game streaming, app-based ride-sharing/transportation and delivery services. In each of these cases, the pool of potential customers is known to the service provider {\em a priori} -- due to appointments that are handed out to patients ahead of time in the healthcare examples, and to subscriptions/sign-ons in the case of streaming and app-based services. However, the finite pool of customers implies that transitory models are nonstationary, both in the sense that they are {\it purely} transient in nature and also because the model parameters can vary temporally. This makes the computation of the performance metrics rather difficult. Consequently, we seek to approximate the performance metric stochastic processes by simpler ones that capture their most vital temporal features. The strong embedding theorems in this paper yield probabilisitic error bounds between the discrete-event performance metric processes and simpler diffusion process approximations in terms of the population size $n$. Our results, therefore, will provide practitioners and engineers with a sense of how large the population size must be, and how the service effort must be concomitantly scaled, so that diffusion approximations can be confidently used for performance analysis and optimization.

Strong approximations were first used for studying time homogeneous queueing models in~\cite{Ro1980}; see the survey paper~\cite{Gl1998} for a comprehensive introduction to the use of strong approximations to $G/G/1$ queueing models in heavy-traffic. In general, the transient analysis of queueing models is rather complicated, and therefore a number of approximations have been developed in appropriate scaling regimes, typically by certain types of reflected diffusion processes~\cite{ChYa2001}. As queueing models can be expressed (approximately) as functionals of random walks, strong approximations are particularly useful in this application context since the driving random walks can be directly replaced by approximating Brownian motion processes. Strong approximation analysis yields rates of convergence and, consequently, rigorous justification of the heavy-traffic approximation on a sample path basis. Our results provide similar insights for a class of nonstationary queueing models under a {\em population acceleration} scaling framework.

We assume that the offered load to the queueing system is `time-of-day' dependent and displays long-range correlations. The modeling and analysis of transitory queues is, in general, quite complicated and we operate under the simplifying assumption that the time-of-day and correlative effects are present solely in the traffic characteristics and that the service requirements of the arriving customers are independent and identically distributed. We propose a two-variable traffic model labeled RS$(G,p)$ wherein we impute the $i$th arriving customer (out of $n$) with the random variable tuple $(T_i,\zeta_i)$, where $T_i$ takes values in $[0,\infty)$ and $\zeta_i$ is binary. $T_i$ models the (potential) arrival epoch of customer $i$ and $\sum_{i=1}^n \zeta_i$ is the number of customers who actually enter the queue; here `RS' stands for `randomly scattered.' We assume that the tuples are independent and identically distributed over the population, and that $T_1$ follows a distribution $G$ and $\mathbb E\zeta_1 = p$. We also assume that the service requirements are generally distributed with finite moment generating function in the neighborhood of zero and independent of the tuple. Consequently we label this the \rsg queue. We make the following contributions in this paper:

\begin{enumerate}
	\item We prove functional strong approximation theorems (FSATs) for the workload and queue length performance metric processes of the \rsg queue in Theorems~\ref{thm:prelim-2} and~\ref{thm:prelim-3} (respectively). These FSATs yield sample path error bounds between the performance metrics and  nonstationary reflected Brownian Bridge processes. The nonstationary Brownian Bridge processes reflect the fact the offered load is time-of-day dependent and has long range correlations.
	\item The proofs of the FSATs are consequences of non-asymptotic functional strong embedding theorems (FSETs) proved for the RS$(G,p)$ traffic process in Proposition~\ref{prop:dropout-strong}, the workload process in Proposition~\ref{prop:vw-sa} and the queue length process in Proposition~\ref{prop:queue-length} that yield exponential probability  bounds as a function of the population size. 
	\item Our proof of the non-asymptotic probabilistic bounds require Dvoretzky-Kiefer-Wolfowitz (DKW) \cite{DKW}  style inequalities for Brownian motion randomly time-changed by a stochastic pure jump process, proved in Proposition~\ref{prop:3}. This generalized DKW inequality may be of independent interest and useful in proving bounds for other types of models.
\end{enumerate}


\paragraph{\textbf{Commentary on main results}}
Our analysis leans on strong approximations for empirical processes and random walks~\cite{CsRe2014}, but also requires substantial innovation. The FSAT in Theorem~\ref{thm:prelim-2} is a consequence of Proposition~\ref{prop:vw-sa} where we prove a strong embedding result for the workload process of a \rsg queue, under the assumption that the service times possess finite moment generating functions in a neighborhood around zero. We show that, with high probability, for a given fixed population size $n$ the sample paths of the workload process can be approximated by those of a reflected Brownian bridge process with time dependent drift and diffusion coefficients. Indeed, we show that the error rate is $O(n^{1/4}\sqrt{\log n})$. Next, the FSAT to the queue length process of the \rsg queue in Theorem~\ref{thm:prelim-3} follows from Proposition~\ref{prop:queue-length}. Paralleling the result in~Proposition~\ref{prop:vw-sa}, we show that the approximating process is a reflected Brownian bridge process with time inhomogeneous drift and diffusion coefficients. However, the drift and diffusion coefficients are scaled versions of those observed in Proposition~\ref{prop:vw-sa}. We note that the analysis of the queue length strong embedding theorem is significantly more involved. The proofs of these results requires a careful construction of a DKW-style inequality for a time-changed Brownian motion process, which we did not find in the literature (see Proposition~\ref{prop:3}). Again, we show that the error rate for the queue length process is $O(n^{1/4}\sqrt{\log n})$.

\paragraph{\textbf{Relation with prior transitory analyses.}}
Note that the RS($G,p$) model affords flexibility for modeling service systems where the pool of potential customers is known {\em a priori}. This typically includes systems where customers `subscribe' to the service ahead of time -- for example, clinics and surgical departments in hospitals where patients are given appointment times, or video and game streaming services with subscribing customers, or ridesharing and food delivery services where the pool of customers are those who have downloaded the smartphone app. In each of these cases, the service provider has knowledge of who the potential customers are, but not all customers will use the service on a given `day'. The randomized arrivals in the RS($G,p$) model accounts for this effect, which is ignored in the \dgi model where $\sum_{i=1}^n \zeta_i = n$ (rendering this variable redundant). Note that the RS($G,p$) model can be extended to a periodic traffic setting, as done in~\cite{GlHo2017}, and the performance metric approximations can still be used in that setting.

The bibliography on the \dgi model now includes pointwise limit results~\cite{Ne2013,Lo1994}, functional strong laws and central limits~\cite{HoJaWa2016a,HoJaWa2016b,GlHo2017,BeHoLe2019} and large deviations principles~\cite{Ho2017,GlHo2017}. In the population acceleration scaling limit, the results in~\cite{HoJaWa2016a,HoJaWa2016b} show that the limiting diffusion for the workload and queue length processes are regulated through a {\em directional derivative} reflection map~\cite{MaRa2010}. This limit can be recovered by the FSATs in Theorems~\ref{thm:prelim-2} and ~\ref{thm:prelim-3}, though the result in~\cite{HoJaWa2016a,HoJaWa2016b} holds under the weaker condition that the service requirements have two finite moments. However, extracting performance measures (such as moments of the workload/queue length) from the directional derivative reflected process is incredibly hard. Indeed, in~\cite{BeHoLe2019}, a different critical scaling is used to show that the queue length converges to a reflected Brownian motion with parabolic drift when the arrival epoch distribution $G$ is exponential. On the other hand, in~\cite{GlHo2017} a special ``critical'' load condition is used to prove that the workload process is approximated by a reflected Brownian motion process. These limit processes can be recovered automatically from the FSATs proved in this paper, albeit at the cost of stronger conditions on the service requirements. We note, however, that with effort it is possible to extend the FSATs to cases where only $m > 2$ moments are available. 

\paragraph{\textbf{Relation to FSATs for nonstationary models.}}
{There is a large and growing literature on nonstationary queueing models covering the whole range of problems that confront the modeling of nonstationary service systems. A crucial difference between this large body of work and the growing literature on transitory models is that the former implicitly assumes an infinite population of customers, while transitory models are exclusively finite population. We cannot possibly do justice to the large body of work on nonstationary models; see~\cite{Wh2018} for a recent review. Instead, we focus on strong approximation results that are most closely related. To the best of our knowledge, strong approximations have been proved almost exclusively for Markovian nonstationary models; note that the literature on strong approximations for stationary queueing networks is far more extensive. The most influential papers in this genre are~\cite{MaMa1995,MaMaRe1998}, where the important {\em uniform acceleration} scaling regime was introduced. In the former, strong approximations for Markov processes were leveraged to prove an FSAT (and consequently functional strong laws and central limit theorems) for an isolated time-varying Markovian single-server queue. This analysis was significantly generalized in the latter paper to include multi-server queueing networks with abandonment. In~\cite{MaPa1995}, strong approximations were leveraged to prove functional limits for state-dependent, nonstationary Markovian queues. More recently~\cite{KoPe2018} consider nonstationary Markovian arrival processes (MAPs) as models of the traffic, and develop a bespoke Poisson representation of the MAP process. They then exploit the strong approximations in~\cite{MaMaRe1998} to prove functional strong laws and central limit theorems. All of these results are premised on the availability of strong approximation results for Markov processes (see, for instance,~\cite[Chapter 7]{EtKu2009}). However, the performance metric processes for the \rsg queue are not Markov (though, of course, one could do state-space enlargement) and we therefore choose to develop the strong approximation results `from scratch.' What is also somewhat remarkable is the fact that we are able to leverage strong approximation results proved for stationary random walks and empirical processes to study nonstationary queueing models without making explicit Markovian assumptions. We believe the methods highlighted in this paper can be used for analyzing other nonstationary stochastic models (such as nonstationary many-server queues, networks of nonstationary queues and even nonstationary multi-class queues).
}


The rest of the paper is organized as follows. We start with preliminaries and main results in Section~\ref{sec:prelim}. In Section~\ref{sec:primer} we provide a brief primer on the strong approximation methodology, particularly the coupling arguments that underly the KMT construction. We do so to make the paper self-contained and since the KMT construction is recondite and not widely understood. Next, we present the DKW-style inequality for controlling the error between the Brownian motion and a counterpart process stochastically time-changed by a jump process in Section~\ref{sec:bm}. Section~\ref{sec:arrive} presents strong embeddings for the RS$(G,p)$ traffic process. The strong embeddings for the workload and queue length processes are proved in Section~\ref{sec:w} and Section~\ref{sec:q} (respectively). We end with commentary and conclusions in Section~\ref{sec:conclusions}.

\section{Preliminaries and Main Results}\label{sec:prelim}
\subsection{A mechanistic model of queueing}\label{sec:model}
Consider a single	 server, infinite buffer queue that is non-preemptive, non-idling, and starts empty. Service follows a first-come-first-served (FCFS) schedule. Let 
$n$ be the {\em nominal} number of customers applying for service. Customers independently sample an arrival epoch $T_i$, $i = 1,\ldots,n$, from a common distribution function. In addition, all customers independently sample identical Bernoulli random variables $\zeta_i$, $i=1,\ldots,n$. Customer $i$ chooses to turn up at time $T_i$ only if $\zeta_i = 1$; we call this the ``dropout'' variable. The arrival process is the cumulative number of customers that have arrived by time $t$. Let \mbox{Bern}($p$) represent the Bernoulli probability distribution with parameter $p$.
\begin{assumption}\label{assum:arriv-dropouts}
For every $n \geq 1$, let $T_1, \ldots, T_n$ be iid samples from a general distribution with distribution function G. 
Denote $G_n$ to be the empirical distribution function given by:
\beq\label{eq:2-d}
G_n(t) := \dfrac{1}{n} \sum_{i=1}^n \mathbf{1}_{\{T_i \leq t\}}.
\eeq
 Let $\zeta_1, \ldots, \zeta_n$ be iid samples from \textnormal{Bern}($p$). Then the arrival process $A_n$ is given by:
\beq\label{eq:arriv-dropouts}
A_n(t) = \sum_{i=1}^{nG_n(t)} \zeta_i.
\eeq
\end{assumption}
\begin{remark}\label{rem:Delta_i}
We call $A_n$ in~\eqref{eq:arriv-dropouts} as the RS$(G,p)$ traffic model. The $\Delta_{(i)}/G/1$ model introduced in \cite{HoJaWa2016a} is a special case of Assumption~\ref{assum:arriv-dropouts}, corresponding to $p=1$.
\end{remark}
\begin{remark}
	Note that it is possible to consider other ways of modeling a random number of arrivals. However, the dropout model considered here is a {\em mechanistic} way of describing the traffic. Note that the model assumes each user will sample a potential time to arrive and a binary indicator that the customer will actually enter the queue at that time. See Section~\ref{sec:conclusions} for further discussion.
\end{remark}
\begin{remark}
	Observe that while the nominal number of arrivals is $n$, the actual number of arrivals realized is random. This traffic model provides a mechanistic description of {\em nonstationary arrivals}: since the distribution $G$ is non-uniform (in general), the expected number of arrivals per-unit time $\mathbf E[A_n(t)]/t$ can be seen to equal $n p G(t)/t$, by an application of Wald's identity. This can be seen as a surrogate of an arrival rate that is clearly time-varying; note that we have not assumed that the distribution is differentiable and consequently defining the rate as the derivative of $\mathbf E[A_n(t)]$ is inappropriate. A crucial point to note is that this time dependency arises from microscopic behavior as opposed to a posited time dependency in the rate function. This stands in contrast with the vast majority of nonstationary models proposed in the literature where the model description starts with posited time-varying rate functions. 
\end{remark}
{Sometimes it is useful to consider arrivals from a general distribution which in turn approaches the limiting distribution $G$ as $n \to \infty$.} 
\begin{assumption}\label{assum:arriv-dropouts-2}
{For every $n \geq 1$, let $T_1, \ldots, T_n$ be iid samples from a general distribution with distribution function $G^{(n)}$ which satisfies the following condition:
\beq\label{eq:Gn-dkw}
 r_n(G) := \sup_{t \in [0, \infty)} \lln G^{(n)}(t) - G(t) \rrn = O\lp \dfrac{1}{\sqrt{n}} \rp,
\eeq
for some strictly increasing and Lipschitz continuous distribution function $G$. In addition, assume that each $G^{(n)}$ is Lipschitz continuous and the Lipschitz coefficient increases at most polynomially in $n$. The arrival process $A_n$ is now defined similar to \eqref{eq:arriv-dropouts}.}
\end{assumption}

\begin{remark}
For simplicity and ease of presentation we will assume that arrivals are supported on $[0, \infty)$, that is, $G(0)= G^{(n)}(0)=0$.
\end{remark}

Next, let $\{V_i, i \geq 1\}$ be a sequence of independent and identically distributed non-negative random variables. $V_i$ represents the service requirement in time units of the $i^{\textnormal{th}}$ potential customer who turns up into the system. We also assume that the sequence is independent of the arrival times $T_i$, $i=1,\ldots,n$ and the corresponding indicators of turning up $\zeta_i$, $i=1, \ldots, n$.
\begin{assumption}\label{assum:1}
For every $n \geq 1$, let $V_1, \ldots, V_n$ be iid samples from a distribution which admits existence of a moment generating function in a neighborhood of zero. Let $\mu$ and $\si$ denote the mean and standard deviation respectively of this distribution. Let 
\beq\label{eq:W_n}
W_n(t) = \sum_{i=1}^{A_n(t)} V_i,
\eeq
denote the cumulative offered load to the system till time $t$.
\end{assumption}
We assume that the server efficiency is $c_n$, i.e. it completes $c_n$ jobs in unit time. Let $M_n(t)$ be the ``truncated'' renewal process counting the number of jobs that the server can complete by time $t$ if working continuously with efficiency $c_n$ (notice that only $n$ jobs arrive to the system):
\beq\label{eq:M_n}
M_n(t) := \sup \lcl 0\leq m \leq n : \sum_{i=1}^m V_i \leq c_n t \rcl.
\eeq

\subsection{Functional strong approximations}
In this section, we list the main results proven in the sequel. Strong approximation results are usually stated in terms of versions of the random variables we wish to approximate. In our case, we require versions of the random arrival times $T_i$, the indicators of turning up $\zeta_i$ and the service times $V_i$. In order to avoid repetition we do not mention this crucial requirement in the following theorem statements. However the same version suffices for each theorem below. Let us also note that is often customary in the literature to assume that the underlying probability space is rich enough to support the random variables as well as the approximating stochastic processes. Our first result provides a strong embedding for the arrival process. Its proof follows from the forthcoming Proposition~\ref{prop:dropout-strong}.
 {\begin{theorem}\label{thm:prelim-1}
There exists a Brownian motion $\hat{B}$, a Brownian bridge $B^{\textnormal{br},n}$ such that if $H_n$ be defined as:
$$
H_n(t) = 
\begin{cases}
npG(t) + \sqrt{n} \lp p B_{G(t)}^{\textnormal{br},n} + \sqrt{p(1-p)} \hat{B}_{G(t)} \rp, &\text{under Assum.~\ref{assum:arriv-dropouts}},\\
np(G(t)+r_n(G)) + \sqrt{n} \lp p B_{G(t)}^{\textnormal{br},n} + \sqrt{p(1-p)} \hat{B}_{G(t)} \rp, &\text{under Assum.~\ref{assum:arriv-dropouts-2},}
\end{cases}
$$
then
$$
\sup_{t \in [0, \infty)} \lln A_n(t) - H_n(t) \rrn \stackrel{\textnormal{a.s.}}{=} O \lp n^{1/4} \sqrt{\log n} \rp.
$$
\end{theorem}}

\begin{remark}
	It is useful to contrast Theorem~\ref{thm:prelim-1} with the setting in \cite{Wh2016}. In the latter, traffic is modeled through a sequence of time-changed stochastic counting processes $\{A^n(t) := (N\circ \Lambda^n)(t)\}$, where $N$ is a stationary stochastic counting process that satisfies an FCLT and $\Lambda^n$ is a posited cumulative arrival rate function that is assumed to be such that $\hat \Lambda_n(t) := n^{-1/2}(\Lambda^n(nt) - nt)$ satisfies $\hat \Lambda_n(t) \to \hat \Lambda(t)$ uniformly on compact sets of $[0,\infty)$ as $n \to \infty$, for some deterministic limit function $\hat \Lambda$. \cite[Theorem 3.1]{Wh2016} shows that the scaled traffic process $\hat A_n(t) := n^{-1/2}(A^n(nt) - nt)$ converges to a limit $B + \hat\Lambda$, where $B$ is a Brownian motion. A vital advantage of such a traffic model is that the stochasticity and the non-stationarities/time-dependencies are completely separated from each other in the limit.
	
	 On the other hand we do not see such a clean separation in $H_n$ immediately. However,  suppose that Assumption~\ref{assum:arriv-dropouts-2} holds with $G(t)=t$ on $[0,1]$, and $n^{1/2}(G^{(n)}(t) - t) \to \hat G(t)$ uniformly on compact sets of $[0,\infty)$ as $n\to\infty$. Then, using the fact that $B^{\textnormal{br},n}_t \stackrel{D}{=} \tilde B^n_t - t \tilde B_1^n$ for a standard Brownian motion process $\tilde B^n$ and ~\cite[Theorem 8.5.2]{Ok2013} it follows that
	\begin{align*}
		\frac{1}{\sqrt{np}} \lp H_n(t) -npt \rp &{=} 
		\dfrac{1}{\sqrt{np}} \lp np\lp G(t) + r_n(G) \rp + \sqrt{n} \lp p B_{G(t)}^{\textnormal{br},n} + \sqrt{p(1-p)} \hat{B}_{G(t)} \rp -npt \rp \\
		&= \sqrt{np} \lp G^{(n)}(t) - t \rp + \sqrt{p} O(1) + \dfrac{1}{\sqrt{p}} \lp p B_{G(t)}^{\textnormal{br},n} + \sqrt{p(1-p)} \hat{B}_{G(t)} \rp\\ 
		&\stackrel{D}{=} \sqrt{np} \lp G^{(n)}(t) - t \rp + \sqrt{p} O(1) + (\sqrt{p} + \sqrt{1-p}) \int_0^t \sqrt{G'(s)} d\bar{B}_s - \sqrt{p} G(t) \bar{B}_1,\\
		&= \sqrt{p}\hat{G}(t) + \sqrt{p}O(1) + \lp \sqrt{p} + \sqrt{1-p} \rp \bar{B}_t - \sqrt{p} t Z.
	\end{align*}
	where $\bar{B}$ is a standard Brownian motion process. Theorem~\ref{thm:prelim-1} immediately shows that, for the arrival process $A_n(t)$ at fixed $t\in[0,1]$, $(np)^{-1/2}(A_n(t) - npt)\Rightarrow \sqrt{p} \hat G(t) + (\sqrt p + \sqrt{(1-p)}) \bar B_t  + \bar Z$ as $n\to\infty$, where $\bar Z$ is a Gaussian random variable with mean $O(\sqrt p)$ and standard deviation $\sqrt{p} t $. This is reminiscent of the limit in~\cite[Theorem 3.1]{Wh2016}, and shows that our framework can recover a separation of the macroscopic time-dependencies and the mesoscopic stochasticity. The setting in~\cite{Wh2016} is important as it forms the basis for a whole series of works around nonstationary queueing models (see the survey~\cite{Wh2018}). We also note that a more rigorous weak limit analysis for a specific choice of $G^{(n)}$ is presented in~\cite{GlHo2017}.
\end{remark}

Our next major result proves strong embeddings for the workload process. In particular, for the cumulative load to the system we have the following result, which follows from the forthcoming Propositions~\ref{prop:workload} and \ref{prop:vw-sa}. 
\begin{theorem}\label{thm:prelim-2}
Along with the Brownian motion $\hat{B}$ and Brownian bridge $B^{\textnormal{br},n}$ as considered in Theorem~\ref{thm:prelim-1}, there exists a Brownian motion $B$ such that if $R_n$ be defined as:
{$$
R_n(t) =  \sqrt{n} \si B_{p G(t)} + \mu {H}_n(t)
$$}
then 
$$
\sup_{t \in [0, \infty)} \lln W_n(t) - R_n(t) \rrn \stackrel{\textnormal{a.s.}}{=} O \lp n^{1/4} \sqrt{\log n} \rp.
$$
Let $\phi$ be the reflection map functional given by
$
\phi(f)(t) := f(t) - \inf_{u \leq t} f(u).
$
Then the total remaining workload at time $t$ can be expressed as $\phi(W_n - c_n \cdot \textnormal{id})(t)$ and this satisfies:
$$
\sup_{t \in [0, \infty)} \lln \phi(W_n - c_n \cdot \textnormal{id})(t) - \phi(R_n - c_n \cdot \textnormal{id})(t) \rrn \stackrel{\textnormal{a.s}}{=} O \lp n^{1/4} \sqrt{\log n} \rp,
$$
where $\textnormal{id}:x \mapsto x$ is the identity map.
\end{theorem}
Finally Theorems~\ref{thm:prelim-1} and \ref{thm:prelim-2} are used to prove a strong embedding for the queue length process, $Q_n$ that includes both any customer in service and all waiting customers. Recall that the queue length $Q_n(t)$ at time $t$ is the difference between the number of arrivals and the number of job completions before time $t$. Denoting by $D_n(t)$ the amount of time the queue stays busy till time $t$, the queue length can be expressed as:
\beq\label{eq:q-length}
Q_n(t) = A_n(t) - M_n(D_n(t)),
\eeq
Finally, the idle time process of the server is given by
\beq\label{eq:idle}
I_n(t) := t-D_n(t).
\eeq
 The following theorem is a consequence of Proposition~\ref{prop:queue-length}.
\begin{theorem}\label{thm:prelim-3}
Let $B$, $\hat{B}$ be the Brownian motions  $B^{\textnormal{br},n}$ the Brownian bridge processes as considered in Theorems ~\ref{thm:prelim-1} and \ref{thm:prelim-2}. Let 
$$
X_n(t) = H_n(t) - \dfrac{c_n t}{\mu} + \sqrt{n} \dfrac{\si}{\mu} B_{{E}_n(t)},
$$
where 
$$
{E}_n(t) =
\begin{cases}
\dfrac{c_n t}{n \mu} + \inf_{s \leq t} \lp pG(s) - \dfrac{c_n}{n}\dfrac{s}{\mu} \rp,~&\text{under Assum.~\ref{assum:arriv-dropouts}},\\
\dfrac{c_n t}{n \mu} + p r_n(G) + \inf_{s \leq t} \lp pG(s) - \dfrac{c_n}{n}\dfrac{s}{\mu} \rp,~&\text{under Assum.~\ref{assum:arriv-dropouts-2}}.
\end{cases} 
$$
Then the queue length $Q_n(t)$ satisfies:
$$
\sup_{t \in [0, \infty)} \lln Q_n(t) - \phi(X_n)(t) \rrn \stackrel{\textnormal{a.s.}}{=} O \lp n^{1/4} \sqrt{\log n} \rp,
$$
if $c_n = O(n^p)$ for some $p>0$ and $\liminf_n c_n > 0$. Else we have
$$
\sup_{t \in [0, \infty)} \lln Q_n(t) - \phi(X_n)(t) \rrn \stackrel{\textnormal{a.s.}}{=} O \lp n^{1/4} \sqrt{\log c_n} \rp.
$$
\end{theorem}

\begin{remark}
Observe that the queue length spends more time near zero as the server efficiency becomes super polynomial in $n$, resulting in a greater approximation error.
\end{remark}

\begin{remark}
	Theorems~\ref{thm:prelim-2} and~\ref{thm:prelim-3} show that the scaled workload process $Z_n := \phi(W_n - c_n \cdot \textnormal{id})/n$ and the scaled queue length process $Q_n/n$ are both closely approximated by nonstationary reflected Brownian motion (RBM) processes on a sample path basis. These theorems also imply the results in~\cite{HoJaWa2016a,HoJaWa2016b,BeHoLe2019} where functional strong laws and central limit theorems were proved for the scaled processes when $p=1$. More importantly, the proofs of Propositions~\ref{prop:workload},~\ref{prop:vw-sa}~and~\ref{prop:queue-length} show that the spatial scale of the processes
\end{remark}

\section{Strong Embeddings: A Primer}\label{sec:primer}
Let $X_1, X_2, \ldots$ be iid random variables from a distribution with mean $0$ and variance $1$. Let $S_n = \sum_{i=1}^n X_i$ denote the $n^{\textnormal{th}}$ partial sum. Then the classical central limit theorem states that:
\beq\label{eq:clt}
\bp \lp \dfrac{S_n}{\sqrt{n}}  \leq y \rp \longrightarrow \Phi(y)~\text{as}~n \to \infty, 
\eeq
where $\Phi$ is the central normal cdf. Equation~\eqref{eq:clt} states that the distribution of $\frac{S_n}{\sqrt{n}}$ approaches that of a standard normal as $n \to \infty$. A stochastic process analog of \eqref{eq:clt} was proved in \cite{donsker}. Let the stochastic process $\{S_n(t); t\in [0,1]\}$ be constructed as follows for each $n \in \N$:
\beq\label{eq:partial-sum-t}
S_n(t) = \dfrac{1}{\sqrt{n}} \lp S_{[nt]} + X_{[nt]+1} + \lp nt - [nt] \rp \rp.
\eeq
Then $\{ S_n(t), t\in [0,1] \}$ converges in distribution to $\{ B(t), t\in [0,1] \}$ as $n \to \infty$, where $B$ is a standard Brownian motion. More precisely,
\beq\label{eq:weak-conv}
h(S_n) \stackrel{d}{\longrightarrow} h(B),
\eeq
for every continuous functional $h:C(0,1) \to \R$. Heuristically, equations~\eqref{eq:partial-sum-t} and \eqref{eq:weak-conv} imply that for $n$ large enough $S_{[nt]} + X_{[nt+1]} (nt - [nt])$ is close in distribution to $\sqrt{n}B_t$. Utilizing the scaling property of Brownian motion and observing that $X_{[nt+1]}$ is negligible compared to $S_{[nt]}$ (for large $
n$), we can concur that $S_k$ is approximately close to $B_k$ for all $k \in \{1,\ldots,n\}$. A bound on the difference of the two was provided in \cite{strassen}, who showed the existence of a probability space containing versions of all associated random variables and processes such that:
\beq\label{eq:strassen}
\dfrac{S_k - B_k}{\sqrt{n \log \log n}} \stackrel{\textnormal{a.s.}}{\longrightarrow}0,~\text{as}~k \to \infty.
\eeq
Equation~\eqref{eq:strassen} can be restated in the following form:
\beq\label{eq:strassen-sup}
\sup_{0 \leq t \leq 1} \dfrac{S_n(t) - \frac{1}{\sqrt{n}}B_{nt}}{\sqrt{\log \log n}} \stackrel{\textnormal{a.s.}}{\longrightarrow} 0.
\eeq

A close associate of the partial sums $S_n$ are the empirical distribution functions corresponding to a sample of iid random variables. Consider for simplicity a random sample $U_1, U_2, \ldots$ of iid $U[0,1]$ random variables. The empirical cdf is then given by:
$$
F_n(t) = \dfrac{1}{n} \sum_{i=1}^n \mathbf{1}_{\{U_i \leq t\}}, ~t \in [0,1].
$$
Observe that the random quantities $\mathbf{1}_{\{U_i \leq t\}}$ are iid with mean $t$ and variance $t(1-t)$. After proper scaling, and considering our previous discussion we expect the empirical process $\al_n$ given by:
$$
\al_n (t) = \sqrt{n} \lp F_n(t) - t \rp,
$$
to be close to a normal random variable with variance $t(1-t)$. We also expect a convergence result akin to \eqref{eq:weak-conv} in the process level. Recall that the standard Brownian bridge $B^{\textnormal{br}}$ is a stochastic process that may be defined as:
$$
B_t^{\textnormal{br}} = B_t - tB_1,~t \in [0,1],
$$
for a Brownian motion $B$. Since $B^{\textnormal{br}}$ is a Gaussian process and $\textnormal{Var}(B_t^{\textnormal{br}}) = t(1-t)$, $B^{\textnormal{br}}$ is a possible candidate for the stochastic process approximating the empirical process. Indeed this was proved to be true in a result analogous to \eqref{eq:strassen-sup} in \cite{brillinger}, who showed the existence of a probability space containing versions of all associated random variables and processes such that:
\beq\label{eq:brillinger}
\sup_{0 \leq t \leq 1} \lln \al_n(t) - B_t^{\textnormal{br}} \rrn \stackrel{a.s.}{=} O \lp {\lp \dfrac{\log n}{n} \rp}^{1/4} {\lp \log n \log \log n \rp}^{1/4} \rp.
\eeq
This result immediately implies the analogue to \eqref{eq:weak-conv}, i.e., $\{ \al_n(t), t\in [0,1] \}$ converges in distribution to $\{ B_t^{\textnormal{br}}, t\in [0,1] \}$.

Equations~\eqref{eq:strassen-sup} and \eqref{eq:brillinger} are insightful and provide a rate of convergence of the partial sums and the empirical processes. However, these are not the best rates of convergence one can achieve. It was shown by \kmt in  \cite{KMT}, that when $X_i$ is allowed to have a finite moment generating function in a neighborhood of $0$:
\beq
\sup_{1\leq k \leq n} \lln S_k - B_k \rrn = O \lp \log n \rp.
\eeq
A similar rate is enjoyed by the empirical processes of uniforms. These two results are stated below in Theorems~\ref{thm:KMT-part-sum} and \ref{thm:KMT-emp-pro}, along with the novel construction (also known as  the Hungarian method) of $X_i$'s and $U_i$'s from the Brownian motion and Brownian bridge respectively. A new and different approach in proving such embedding results has been provided in \cite{chatterjee} for the simple symmetric random walk. We will use the terminology \emph{strong embedding} for coupling an arbitrary random variable $W$ with a Gaussian random variable $Z$ so that $W-Z$ has exponentially decaying tails at the appropriate scale. Theorems~\ref{thm:KMT-part-sum} and \ref{thm:KMT-emp-pro} thus provide strong embeddings to the partial sums $S_n$ and the empirical processes $\al_n$. As alluded to in the Introduction we will apply these results to obtain strong embeddings for the performance metrics of a $\textnormal{RS}(G,p)/G/1$ queue.
 
 \subsection{Strong embedding of the random walk}
 We present the KMT theorem for the strong embedding of the random walk. Proof ideas and construction can be found in the Appendix.
\begin{theorem}\label{thm:KMT-part-sum}
Let $F$ be a distribution function with mean $0$ and variance $1$. In addition, suppose the moment generating function corresponding to $F$, $R(t) = \E (e^{tX})$, $X \sim F$, exists in a neighborhood of $0$. Then, given a Brownian motion $B$, and using it, one can construct a sequence of random variables $X_1, X_2, \ldots$ which are independent and identically distributed to $F$. Furthermore, the partial sums of $X_i$'s are strongly coupled to the Brownian motion $B$ in the following sense. For every $n \in \N$ and $x > 0$:
\beq\label{eq:strong-approx-1}
\bp \lp \sup_{1\leq k\leq n} \lln \sum_{i=1}^n X_i - B_n \rrn > C \log n + x \rp < K e^{-\la x},
\eeq
where $C$, $K$ and $\la$ are positive constants depending only on $F$.
\end{theorem}

\subsection{Strong embedding of the empirical process}
We present the strong embedding result for the empirical process. Proof ideas and construction can be found in the Appendix.
\begin{theorem}\label{thm:KMT-emp-pro}
There exists a probability space with independent $U[0,1]$ random variables  $U_1, U_2, \ldots$ and a sequence of Brownian bridges $B_1^{\textnormal{br}}, B_2^{\textnormal{br}}, \ldots$ such that for all $n \geq 1$ and $x \in \R$:
\beq\label{eq:strong-approx-2}
\bp \lp \sup_{s \in [0,1]} \sqrt{n} \lln \al_n(s) - B_n^{br}(s) \rrn > C \log n + x \rp < K e^{-\la x},
\eeq
{for some constants $C$, $K$ and $\la$.} Here the empirical process $\al_n$ is given by
$$
\al_n (s) = \sqrt{n} (F_n(s) - s), 
$$
and 
$$
F_n(s) = \dfrac{1}{n} \sum_{i=1}^n \mathbf{1}_{\{U_i \leq s\}}.
$$
\end{theorem}
{
\begin{remark}\label{rem:KMT-emp-pro-const}
The constants $C$, $K$ and $\la$ in Theorem~\ref{thm:KMT-emp-pro} can be chosen as $C=100$, $K=10$ and $\la= 1/50$. See \cite[Theorem 4.4.1]{CsRe2014} for more details.
\end{remark}
}

\begin{remark}\label{rem:bridge-diff-n}
The KMT construction relies on the generation of a sample of $n$ uniforms $U_1, \ldots, U_n$ from a Brownian bridge $B^{\textnormal{br},n}$. It can be seen from the construction that having obtained $\{U_1, \ldots, U_n\}$, one is unable to obtain another $U_{n+1}$ such that the new set $\{U_1, \ldots, U_{n+1}\}$ satisfies \eqref{eq:strong-approx-2} with the same Brownian bridge. Instead it would be necessary to redo the construction. This necessitates the need for a different Brownian bridge $B^{\textnormal{br},n}$ for every $n$.
\end{remark}

\section{Control of Time-changed Brownian motion}\label{sec:bm}
Our analyses in subsequent sections provide strong embedding results for several queue length characteristics to corresponding diffusion approximations. In order to achieve those results we need a strong control over the difference between a Brownian motion evaluated at several $n$-level stochastic quantities and their corresponding fluid limits as $n$ goes to infinity (for example, the empirical distribution of arrival epochs against the true arrival distribution). In this section we present general results on bounding the difference between Brownian motion evaluated at some stochastic jump process and its fluid limit. The forthcoming Proposition~\ref{prop:3} is rather general and might be of independent interest. We start by stating an assumption on the fluid limit.
\begin{assumption}\label{assum:tBM-1}
For each $n \geq 1$, let $\xi_n:[0,\infty) \mapsto \R$ be a bounded Lipschitz continuous function; that is, there exists $c_{\xi_n}>0$ such that 
$$
\lln \xi_n(s) - \xi_n(t) \rrn \leq c_{\xi_n} \lln s-t \rrn,
$$
for all $s, t \in [0, \infty)$. 
\end{assumption}
We also impose regularity conditions on the stochastic jump process along which our Brownian motion will be evaluated. These are collected in the following assumption.
\begin{assumption}\label{assum:tBM-2}
Let ${\{T_n\}}_{n \geq 1}$ be a sequence of non-decreasing positive numbers. Let $\Xi_n :=  \{ \Xi_n(s); s \in [0,L_n]\}$ be a stochastic pure jump process defined on a common probability space $(\Omega, \cf, \bp)$ for every $n \geq 1$, such that almost surely the number of its jumps in $[0,L_n]$ is bounded above by $kn^m$ for some fixed positive constants $k$ and $m$. In addition, assume that:
$$
\limsup_{n} \sup_{s \in [0,L_n]} \lcl |\Xi_n(s)|\rcl < \infty.
$$
Denote $D = \limsup_n \sup_s |\Xi_n(s)|$.
\end{assumption}
In order to obtain a non-asymptotic probabilistic bound on the difference $|B_{\Xi_n(s)} - B_{\xi_n(s)}|$, where $\xi_n$ and $\Xi_n$ are introduced in Assumptions~\ref{assum:tBM-1} and \ref{assum:tBM-2} respectively, we impose further conditions on the distribution of $|\Xi_n-\xi_n|$. In particular, we require a Dvoretzky-Kiefer-Wolfowitz (DKW) style inequality~\cite{DKW} for the tail distribution of $|\Xi_n-\xi_n|$.
\begin{assumption}\label{assum:tBM-3}
For every $n \geq 1$, let $\xi_n$ and $\Xi_n$ be as considered in Assumptions~\ref{assum:tBM-1} and \ref{assum:tBM-2}. Let there be constants $k_0$, $k_1$, $k_2$, $k_3$ and $0<\ga<4$ such that the following inequality holds for every $\ep >0$:
$$
\bp \lp \sup_{s \in [0,L_n]} \lln \Xi_n(s) - \xi_n(s) \rrn > \ep + k_0 \dfrac{\log n}{n} \rp \leq k_1 e^{-k_2 n {^{\ga}} \ep^2 \wedge k_3 n^{\ga} \ep}.
$$
\end{assumption}
In addition, denote $\al_n$ by:
\beq\label{eq:al_n}
\al_n :=\dfrac{1}{\sqrt{2}} {\lp \sup_{t \in [0,L_n]}|\Xi_n(t)-\xi_n(t)| \rp}^{1/2}.
\eeq
In the forthcoming Proposition~\ref{lem:tBM-satisfy-1} and Proposition~\ref{lem:tBM-satisfy-2} we will show that the Assumptions~\ref{assum:tBM-1}, \ref{assum:tBM-2} and \ref{assum:tBM-3} are satisfied for the arrival process given in \eqref{eq:arriv-dropouts} and the truncated renewal process given in \eqref{eq:M_n}. In order to prove these two lemma's, we first recall a few facts on sub-exponential random variables.
\begin{lemma}\label{lem:subexp-1}
Let $X_1, \ldots, X_n$ be iid copies of a random variable with mean $\mu$ such that there exist parameters $(\nu, m)$ satisfying:
\beq\label{eq:subexp-def}
\bex\lc e^{\la (X-\mu)} \rc \leq e^{\frac{\nu^2 \la^2}{2}}~\text{ for all }|\la| < \dfrac{1}{m}.
\eeq
Then the following holds true:
\begin{enumerate}
\item[(i)] 
\beq\label{eq:subexp-1}
\bp \lp \lln \sum_{i=1}^n X_i - n\mu \rrn \geq nt \rp \leq
\begin{cases}
2 e^{-\frac{nt^2}{2 \nu^2}}~\text{ for } 0\leq t \leq \dfrac{\nu^2}{m},\\
2 e^{-\frac{nt}{2m}}~\text{ for }t>\dfrac{\nu^2}{m}.
\end{cases} 
\eeq

\item[(ii)]
\beq\label{eq:subexp-2}
\bp \lp \sup_{0 \leq k \leq n} \lln \sum_{i=1}^k X_i - k \mu \rrn \geq nt \rp \leq
\begin{cases}
2 e^{-\frac{nt^2}{2 \nu^2}}~\text{ for } 0\leq t \leq \dfrac{\nu^2}{m},\\
2 e^{-\frac{nt}{2m}}~\text{ for }t>\dfrac{\nu^2}{m}.
\end{cases}
\eeq
\end{enumerate}
\end{lemma}

\begin{proof}
\noindent
\emph{(i)}
The result follows from the usual considerations for sub-exponential random variables (see e.g. \cite[Section~2.1.3]{Wainwright}). The main ingredient is a Chernoff-type approach to obtain:
\begin{align}\label{eq:subexp-3}
\bp \lp \sum_{i=1}^n X_i - n\mu \geq nt \rp &\leq \dfrac{\bex[e^{\la (\sum_{i=1}^n X_i - n \mu)}]}{e^{n \la t}} \\
&\leq \exp \lp -n \la t + \dfrac{n \la^2 \nu^2}{2} \rp. \nonumber
\end{align}
Optimization of the right hand side followed by a premultiplication by $2$ to obtain the two-sided tail bound yields the desired result \eqref{eq:subexp-1}.

\noindent
\emph{(ii)}
Observe that $M_k = \sum_{i=1}^k X_i - k \mu$ is a martingale. In addition, $x \mapsto e^{\la x}$ is a convex function. Consequently $e^{\la M_k}$ is submartingale. Thus applying Doob's martingale inequality we obtain:
\begin{align*}
\bp \lp \sup_{0 \leq k \leq n} \sum_{i=1}^k X_i - k \mu \geq nt \rp = \bp \lp \sup_{0 \leq k \leq n} e^{\la (\sum_{i=1}^k X_i - k \mu)} \geq e^{\la nt} \rp\\
\leq \dfrac{E \lp e^{\la (\sum_{i=1}^k X_i - k\mu)} \rp}{e^{\la n t}},
\end{align*}
thus reducing our considerations to \eqref{eq:subexp-3}. The same arguments carry forward and we obtain \eqref{eq:subexp-2}.
 \end{proof}
The assumption~\eqref{eq:subexp-def} in Lemma~\ref{lem:subexp-1} holds for every random variable $X$ with a finite moment generating function in a neighborhood of $0$. This is a consequence of the following lemma.
\begin{lemma}\label{lem:subexp-2}
Let $X$ be a random variable with mean $\mu$, whose moment generating function exists in a neighborhood of $0$. Then we have:
\beq\label{eq:subexp-4}
\bex\lc e^{\la (X-\mu)} \rc \leq e^{\la^2 \nu^2/2}~\text{ for all }|\la| < \dfrac{1}{m},
\eeq
where $\nu = \sqrt{2 \textnormal{Var}(X)}$ and $m$ is given by the condition
$$
\bex\lc e^{2 \la |X-\mu|} \rc < 4 \text{ for all }|\la| < \dfrac{1}{m}.
$$
\end{lemma}

\begin{proof}
Observe that the moment generating function of $X$ satisfies:
$$
\bex\lc e^{\la (X-\mu)} \rc \leq \bex\lp 1 + \la(X-\mu) + \dfrac{\la^2}{2} (X-\mu)^2 e^{\la |X-\mu|} \rp.
$$
Noticing $\bex X=\mu$ and by Cauchy-Schwarz inequality we have:
\begin{align*}
\bex\lc e^{\la (X-\mu)} \rc &\leq 1 + \dfrac{\la^2}{2} \sqrt{\bex (X-\mu)^4} \sqrt{\bex \lc e^{2 \la |X-\mu|} \rc}\\
&\leq 1+ \la^2 \textnormal{Var}(X) \sqrt{\dfrac{\bex[e^{2 \la |X-\mu|}]}{4}}.
\end{align*}
Hence for all $\la$ satisfying $\bex[e^{2 \la |X-\mu|}] < 4$ we have:
$$
\bex[e^{\la (X-\mu)}] \leq 1+\la^2 \textnormal{Var} (X).
$$
However, $1+\la^2 \textnormal{Var}(X) \leq e^{\la^2 \textnormal{Var}(X)}$ and consequently:
$$
\bex\lc e^{\la(X-\mu)} \rc \leq e^{\la^2 \textnormal{Var}(X)}.
$$ 
This yields \eqref{eq:subexp-4}.
\end{proof}

\begin{proposition}\label{lem:tBM-satisfy-1}
Let Assumption~\ref{assum:arriv-dropouts} hold with {$G$ being the uniform distribution function on $[0,1]$}. Let $\hat{A}_n$ denote the corresponding arrival process in \eqref{eq:arriv-dropouts}. Then Assumptions~\ref{assum:tBM-1}, \ref{assum:tBM-2} and \ref{assum:tBM-3} hold with
$$
\xi_n(t) = pt~ \text{ and } ~\Xi_n(t) = \dfrac{\hat{A}_n(t)}{n}, \text{ for } t \in [0,1],
$$
and $L_n = 1$ for all $n$.
\end{proposition}
\begin{proof}
Note that $\xi_n$ is bounded and Lipschitz with $c_{\xi_n} = p$. Next, notice that $\hat{A}_n$ has at most $n$ jumps in $[0,1]$. In addition:
$$
\sup_{s \in [0,1]} \lcl \lln \dfrac{\hat{A}_n(s)}{n} \rrn, \lln ps \rrn \rcl \leq 1.
$$
It remains to prove a DKW type inequality for the difference $|\frac{\hat{A}_n(t)}{n} - pt|$.  To that effect observe that 
\begin{align*}
\dfrac{\hat{A}_n(s)}{n} - ps &= \dfrac{\sqrt{p(1-p)}}{n} \sum_{i=1}^{nF_n(s)} \dfrac{\chi_i - p}{\sqrt{p(1-p)}} + p\lp F_n(s) -s \rp\\
&= \sqrt{p(1-p)} \dfrac{1}{n} \sum_{i=1}^{n F_n(s)} X_i + p\lp F_n(s) -s \rp,
\end{align*}
where $X_i = \frac{\chi_i - p}{\sqrt{p(1-p)}}$ are iid random variables with mean $0$ and variance $1$. 
Consequently, we have 
\beq\label{eq:star1}
\bp \lp \sup_{s \in [0,1]} \lln \dfrac{\hat{A}_n(s)}{n} - ps \rrn > \ep \rp \leq \bp \lp \sqrt{p(1-p)} \sup_{0\leq k \leq n} \dfrac{1}{n} \lln \sum_{i=1}^k X_i \rrn > \dfrac{\ep}{2} \rp + \bp \lp p \lln F_n(s) -s \rrn > \dfrac{\ep}{2} \rp.
\eeq
From the standard DKW inequality for empirical distributions~\cite{DKW}, the second term has the standard exponentially decreasing bound given by:
\beq\label{eq:star2}
\bp \lp \sup_{s \in [0,1]} p \lln F_n(s) - s \rrn > \dfrac{\ep}{2} \rp \leq 2 e^{- n \ep^2 /(4p^2)}.  
\eeq
For the first term, observe that the $X_i$'s have a finite moment generating function $E[e^{\la X_i}]$ for all $\la$. Hence appealing to Lemma~\ref{lem:subexp-1} and Lemma~\ref{lem:subexp-2} we obtain:
\beq\label{eq:subexp-dkw-1}
\bp \lp \sup_{0 \leq k \leq n} \lln \sum_{i=1}^k X_i \rrn > \dfrac{n \ep}{2 \sqrt{p(1-p)}} \rp \leq 2 e^{-\frac{n \ep^2}{8p(1-p)}}.
\eeq
Consequently there exist constants $k_1$ and $k_2$ such that
$$
\bp \lp \sup_{s \in [0,1]} \lln \dfrac{\hat{A}_n(s)}{n} - ps  \rrn > \ep \rp \leq k_1 e^{-k_2 n \ep^2}.
$$
\end{proof}

\begin{proposition}\label{lem:tBM-satisfy-2}
Let Assumption~\ref{assum:1} hold and $M_n$ be given by \eqref{eq:M_n}. Then for any sequence of non-decreasing positive reals $L_n$, Assumptions~\ref{assum:tBM-1}, \ref{assum:tBM-2} and \ref{assum:tBM-3} hold with
$$
\xi_n(t) = \lp \dfrac{c_n}{n} \dfrac{t}{\mu} \rp\wedge 1 ~\text{ and }~ \Xi_n(t) = \dfrac{M_n(t)}{n},~\text{ for }t\in [0, L_n].
$$
\end{proposition}
\begin{proof}
Note that $\xi_n$ is bounded and Lipschitz with $c_{\xi_n} = \frac{c_n}{n \mu}$. Next notice that $M_n$ has at most $n$ jumps in $[0,L_n]$. In addition
$$
\sup_{s \in [0,L_n]}  \lln \dfrac{M_n(s)}{n} \rrn < 1.
$$
It remains to prove a DKW-style inequality for the difference $|\frac{M_n(t)}{n} - (\frac{c_n}{n}\frac{t}{\mu})\wedge 1|$. Observe that for any $L_n$ positive:
\begin{align*}
\sup_{0 \leq t \leq L_n} \lln  \dfrac{M_n(t)}{n} - \lp \dfrac{c_n}{n} \dfrac{t}{\mu} \rp \wedge 1 \rrn &\leq \sup_{0 \leq t \leq \frac{S_n}{c_n}+} \lln \dfrac{M_n(t)}{n} - \dfrac{c_n}{n} \dfrac{t}{\mu} \rrn + \sup_{\frac{S_n}{c_n} \leq t \leq L_n} \lln 1- \lp \dfrac{c_n}{n}\dfrac{t}{\mu} \rp \wedge 1 \rrn\\
&\leq \sup_{0 \leq t \leq \frac{S_n}{c_n}+} \dfrac{\lln N(c_nt)-\frac{c_n t}{\mu} \rrn}{n} + \dfrac{1}{n} + \dfrac{\lln \frac{S_n}{n} - \mu \rrn}{\mu},
\end{align*}
where $N(t) = \inf \{ m \geq 0: \sum_{i=1}^m V_i > t \}$. By a change of variable the first term on the right hand side has a simpler representation upon which we have:
\beq\label{eq:str-app-ineq-1}
\sup_{0 \leq t \leq L_n} \lln  \dfrac{M_n(t)}{n} - \lp \dfrac{c_n}{n}\dfrac{t}{\mu} \rp \wedge 1 \rrn \leq \sup_{0 \leq s \leq \frac{S_n}{\mu}+} \dfrac{|\tilde{N}_s -s|}{n} + \dfrac{1}{n} + \dfrac{|\frac{S_n}{n}-\mu|}{\mu},
\eeq
where $\tilde{N}_s = \inf \{ m\geq 0: \sum_{i=1}^m \frac{V_i}{\mu} > t \}$. 
From \cite[Lemma]{Ho84} and observing that $\tilde{N}(\frac{S_n}{\mu}+) = n$, we notice that
\beq\label{eq:str-app-ineq-2}
\sup_{0 \leq s \leq \frac{S_n}{\mu}+} \lln \tilde{N}(s) - s \rrn \leq \sup_{0 \leq s \leq n} \lln \tilde{U}(s) - s \rrn, 
\eeq
where $\tilde{U}(s) = \sum_{i=1}^{[s]} \frac{V_i}{\mu}$. In addition 
\beq\label{eq:str-app-ineq-3}
\sup_{0 \leq s \leq n} \lln \tilde{U}(s) - s \rrn \leq \sup_{0\leq k \leq n} \lln \tilde{S}_k - k \rrn + 1,
\eeq
where $\tilde{S}_k = \sum_{i=1}^k \frac{V_i}{\mu}$.
Combining equations \eqref{eq:str-app-ineq-1}, \eqref{eq:str-app-ineq-2} and \eqref{eq:str-app-ineq-3} we obtain for all $\ep > 0$:
\begin{multline}\label{eq:str-app-ineq-4}
\bp \lp \sup_{0 \leq t \leq L_n} \lln \dfrac{M_n(t)}{n} - \lp \dfrac{c_n}{n}\dfrac{t}{\mu} \rp \wedge 1 \rrn \geq \ep + \dfrac{2}{n} \rp \\
\leq \bp \lp \sup_{0 \leq k \leq n} \lln  \tilde{S}_k - k\rrn > \dfrac{\ep n}{2} \rp + \bp \lp \lln \dfrac{S_n}{n} - \mu \rrn > \dfrac{\ep \mu}{2} \rp.
\end{multline}
Observe that $V_i$'s have a finite moment generating function in a neighborhood of $0$. Hence appealing to Lemma~\ref{lem:subexp-1} and Lemma~\ref{lem:subexp-2} we have for all $\ep > 0$:
$$
\bp \lp \sup_{0 \leq k \leq n} \lln \tilde{S}_k - k \rrn > \dfrac{n \ep}{2} \rp  \leq 2 \exp \lp -k_2' n \ep^2 \wedge k_3' n \ep \rp,
$$
for some constants $k_2$ and $k_3$. Similarly, Lemma~\ref{lem:subexp-1} and Lemma~\ref{lem:subexp-2} also imply for all $\ep > 0$:
$$
\bp \lp \lln \dfrac{S_n}{n} - \mu \rrn > \dfrac{\ep \mu}{2} \rp  \leq 2 \exp \lp -k_4' n \ep^2 \wedge k_5' n \ep \rp.
$$
Consequently we have constants $k_1$, $k_2$ and $k_3$ such that for all $\ep > 0$:
$$
\bp \lp \sup_{0 \leq t \leq L_n} \lln \dfrac{M_n(t)}{n} - \lp \dfrac{c_n}{n} \dfrac{t}{\mu} \rp \wedge 1 \rrn \geq \ep +\dfrac{2}{n}  \rp  \leq k_1 e^{-k_2 n \ep^2 \wedge k_3 n\ep}.
$$
\end{proof}

In the following lemma we obtain an upper bound on the expected value of $\al_n$ as denoted in \eqref{eq:al_n}. This result would be utilized in the forthcoming Lemma~\ref{lem:ga_n}.
\begin{lemma}\label{lem:al_n}
Let Assumptions \ref{assum:tBM-1}, \ref{assum:tBM-2} and \ref{assum:tBM-3} hold. Then there exists a constant $C'$ such that:
\beq\label{eq:E-al_n}
\bex\lc \al_n \rc \leq \dfrac{C'}{n^{\ga/4}}.
\eeq
\end{lemma}

\begin{proof}
Since $\xi_n$ is bounded, without loss of generality let us assume
$$
\sup_{s \in [0,L_n]} \lcl |\xi_n(s)|,  \lln \Xi_n(s) \rrn \rcl \leq D.
$$
As a consequence notice:
$$
\al_n = \frac{1}{\sqrt{2}} { \lp \sup_{t \in [0,L_n]} \lln \Xi_n(t) - \xi_n(t) \rrn \rp}^{1/2} \leq \sqrt{D}.
$$ 
Thus we have:
\begin{align}\label{eq:ub-BTIS-2}
\bex\lc \al_n \rc &= \int_0^{\sqrt{D}} \bp [\al_n > t] dt\nonumber = \int_{0}^{\sqrt{D}} \bp \lc \sup_{t \in [0, L_n]} \lln \Xi_n(t) - \xi_n(t) \rrn > 2t^2 \rc dt \\
&= \dfrac{1}{2\sqrt{2}} \int_0^{2D}  \dfrac{1}{\sqrt{s}} \bp \lc \sup_{t \in [0, L_n]} \lln \Xi_n(t) - \xi_n(t) \rrn > s \rc ds,
\end{align}
where the last step is obtained by a change of variable. Now breaking the integral into parts and using Assumption~\ref{assum:tBM-3} we obtain from above:
\begin{align*}
\bex& \lc \al_n \rc\leq \dfrac{1}{2\sqrt{2}} \int_0^{k_0 \frac{\log n}{n}} \dfrac{1}{\sqrt{s}} ds + \int_{k_0 \frac{\log n}{n}}^{2D} \dfrac{1}{\sqrt{s}} \bp \lp \sup_{t \in [0, L_n]} \lln \Xi_n(t) -\xi_n(t) \rrn > s \rp ds\\
&\leq \dfrac{1}{\sqrt{2}} \sqrt{ k_0 \dfrac{\log n}{n} } +  \dfrac{1}{2\sqrt{2}} \int_0^{2D - k_0 \frac{\log n}{n}} \dfrac{1}{\sqrt{s + k_0 \frac{\log n}{n}}} \bp \lp \sup_{t \in [0, L_n]} \lln \Xi_n(t) - \xi_n(t) \rrn > k_0 \dfrac{\log n}{n} + s \rp ds \\
&\leq \dfrac{1}{\sqrt{2}} \sqrt{ k_0 \dfrac{\log n}{n} }  + \dfrac{1}{2 \sqrt{2}} \int_0^{2D - k_0 \frac{\log n}{n}}\dfrac{1}{\sqrt{s}} k_1 e^{-k_2n^{\ga}s^2 \wedge k_3 n^{\ga} s} ds\\
&\leq \dfrac{1}{\sqrt{2}} \sqrt{ k_0 \dfrac{\log n}{n} } + \dfrac{1}{4\sqrt{2} n^{\ga/4}} \int_0^{\infty} k_1 e^{-k_2 z \wedge k_3 \sqrt{z} n^{\ga /2}}z^{-3/4} dz \leq C' \dfrac{\sqrt{\log n}}{n^{\ga/4}},
\end{align*}
for some constant $C'$, where the penultimate step is obtained by a change of variable ($n^{\ga} s^2 \mapsto z$) and then use of the fact that the integral $\int_0^{\infty} \exp (-4 k_2 z \wedge k_3 \sqrt{z} n^{\ga/2}) z^{-3/4} dz$ is bounded yields our desired result \eqref{eq:E-al_n}.
\end{proof}
\begin{assumption}\label{assum:tBM-4}
Let $\xi_n$ and $\Xi_n$ be as considered in Assumptions~\ref{assum:tBM-1} and \ref{assum:tBM-2}. Let a standard Brownian motion $B$ be defined on the same probability space $({\Omega}, {\cf}, \bp)$. Then let $\hat{f_s}$ be the Gaussian process defined on $[0,L_n]$ by 
$$
\hat{f_s} = B_{\xi_n(s)} - B_{\Xi_n(s)}.
$$
In addition, let $\bP$ denote the conditional probability on $(\Omega, \cf, \bp)$ given $\Xi_n$. Thus the conditional expectation $\E [Z] = \bex [Z \vert \Xi_n(s); s \in [0, L_n]]$. Denote $\ga_n := \bex [\sup_{t \in [0,L_n]} \hat{f}_t ]$.
\end{assumption}

The key ingredient to find probabilistic bounds for $\hat{f}$ as alluded to above is the Borell-TIS inequality included below for completeness. As a first step we find the conditional expectation of $\sup \hat{f}$ given $\Xi_n$. This is obtained in the following lemma. 
\begin{lemma}
Let Assumptions~\ref{assum:tBM-1}, \ref{assum:tBM-2} and \ref{assum:tBM-4} hold. Then there exist constants $M$ and $\tilde{C}$ such that 
\beq\label{eq:cond-E}
\E \lc  \sup_{t \in [0,L_n]} \hat{f}_t \rc \leq M \lp {\int_0^{\al_n} \sqrt{\log \lp \dfrac{c_{\xi_n} L_n}{\ep^2} + 1 \rp} d\ep + \tilde{C} \al_n \sqrt{\log n}} \rp
\eeq
\end{lemma}
\begin{proof}
The canonical metric for $\hat{f}$ in $(\hat{\Omega}, \hat{\cf}, \bP)$ is given by
\beq\label{eq:can-metr-1}
\hat{d}(s,t) = {\lp \E \lc (\hat{f_s} - \hat{f_t})^2 \rc \rp}^{1/2}.
\eeq
Let $\hat{D}$ denote the diameter of $[0,L_n]$ with respect to the canonical metric, i.e.,
$$
\hat{D} = \sup_{s,t \in [0,L_n]} \hat{d}(s,t).
$$
Let $\hat{N}(\ep)$ be the metric entropy defined by the smallest number of balls of diameter $\ep$ (with respect to the canonical metric $\hat{d}$) that cover $[0,L_n]$. Then from \cite[Theorem 1.3.3]{adler} there exists a universal constant $M$ such that 
\beq\label{eq:E-sup-ub}
\E \lc \sup_{t \in [0,L_n]} \hat{f_t} \rc \leq M \int_0^{\hat{D}/2} {\lp \log \hat{N}(\ep) \rp}^{1/2} d\ep.
\eeq
It can be easily shown that the canonical metric $\hat{d}$ as defined in \eqref{eq:can-metr-1} satisfies:
\begin{multline*}
\hat{d}(s,t)^2 =  \lc \lln \xi_n(s)-\Xi_n(s) \rrn + \lln \xi_n(t)-\Xi_n(t) \rrn -2 (\xi_n(s) \wedge \xi_n(t) + \Xi_n(s) \wedge \Xi_n(t) \right.\\
 \left.- \xi_n(s) \wedge \Xi_n(t) -  \xi_n(t) \wedge \Xi_n(s)) \rc.
\end{multline*}
If the numbers $\xi_n(s),\Xi_n(s),\xi_n(t),\Xi_n(t)$ be arranged in ascending order form the vector $(d_1,d_2,d_3,d_4)$, then it can be shown that 
\beq
\hat{d}(s,t) = \sqrt{(d_4-d_3) + (d_2 - d_1)}.
\eeq
This implies, for $s$ and $t$ such that $\Xi_n(s)=\Xi_n(t)$, 
$$
\hat{d}(s,t) = \sqrt{|\xi_n(s)- \xi_n(t)|} \leq \sqrt{c_{\xi_n}} \sqrt{|s-t|}.
$$ 
In addition, note that ${\hat{D}}/2=\sup_{s,t \in [0,L_n]} \hat{d}(s,t)/2 \leq (\sup_{t \in [0,L_n]}|\Xi_n(t)-\xi_n(t)|)^{1/2}/\sqrt{2} = \al_n$.

In order to obtain an upper bound to $\hat{N}(\ep)$, recall as mentioned earlier, $\hat{d}(s,t) \leq \sqrt{c_{\xi_n}} \sqrt{|s-t|}$ whenever $\Xi_n(s)=\Xi_n(t)$. Since $\Xi_n$ has at most $kn^m$ points of discontinuity, there are at most $(kn^m+1)$ intervals where $\Xi_n$ is constant. Let these intervals be $R_0,\ldots,R_{kn^m}$. Then $\hat{N}(\ep)$ can be bounded above as follows:
\beq\label{eq:N(ep)-ub}
\hat{N}(\ep) \leq \sum_{i=0}^{kn^m} \left\lceil{c_{\xi_n}\dfrac{R_i}{\ep^2}}\right\rceil \leq \sum_{i=0}^{kn^m} c_{\xi_n}\dfrac{R_i}{\ep^2} + (kn^m+1) =c_{\xi_n} \frac{L_n}{\ep^2} + (kn^m+1).
\eeq
Thus, using \eqref{eq:N(ep)-ub} we get from \eqref{eq:E-sup-ub}:
\beq\label{eq:Borell-TIS-2}
\E \lc \sup_{t \in [0,L_n]} \hat{f}_t \rc \leq M \int_0^{\al_n} \sqrt{\log \lp \dfrac{c_{\xi_n}L_n'}{\ep^2}+kn^m+1 \rp} d\ep.
\eeq
Observe that $\log(x+y) \leq \log(x+1) + \log(y)$ for $x \geq 0$ and $y \geq 1$. Consequently we obtain
\begin{align}\label{eq:int-sqrtlog-ub}
\int_0^{\al_n} \sqrt{\log \lp \frac{c_{\xi_n}L_n}{\ep^2}+kn^m+1 \rp} d\ep &\leq \int_0^{\al_n} \sqrt{\log \lp \dfrac{c_{\xi_n}L_n}{\ep^2} + 1 \rp + \log(kn^m+1)} d\ep \nonumber\\
&\leq \int_0^{\al_n} \sqrt{\log \lp \dfrac{c_{\xi_n} L_n}{\ep^2} + 1 \rp} d\ep + \al_n \sqrt{\log(kn^m+1)},
\end{align}
where in the last step we have used the fact that $\sqrt{x+y} \leq \sqrt{x} + \sqrt{y}$. From \eqref{eq:Borell-TIS-2} and \eqref{eq:int-sqrtlog-ub} we obtain that there exists a constant $\tilde{C}$ such that:
$$
\E \lc  \sup_{t \in [0,L_n]} \hat{f}_t \rc \leq M \lp {\int_0^{\al_n} \sqrt{\log \lp \dfrac{c_{\xi_n}L_n}{\ep^2} + 1 \rp} d\ep + \tilde{C} \al_n \sqrt{\log n}} \rp.
$$
\end{proof}
Having obtained the conditional expectation of $\sup \hat{f}$, we next obtain the unconditional expectation of $\sup \hat{f}$. This is achieved in the following result and would be a crucial ingredient in the proof of the forthcoming Proposition~\ref{prop:3}.
\begin{lemma}\label{lem:ga_n}
Let Assumptions \ref{assum:tBM-1}, \ref{assum:tBM-2}, \ref{assum:tBM-3} and \ref{assum:tBM-4} hold. Then there exists constant $C$ such that
$$
\ga_n = \bex \lc \sup_{t \in [0,L_n]} \hat{f}_t \rc \leq C \dfrac{\sqrt{\log(L_n \vee n)}}{n^{\ga / 4}}.
$$
\end{lemma}

\begin{proof}
Using integration by parts, and denoting $c_{\xi_n}L_n$ by $L_n'$, the first term in \eqref{eq:cond-E} yields
\begin{align}\label{eq:int-sqrtlog-ub-2}
\int_0^{\al_n} \sqrt{\log \lp \dfrac{L_n'}{\ep^2} + 1 \rp} d\ep &= \al_n \sqrt{\log \lp \dfrac{L_n'}{\al_n^2} + 1 \rp} + \int_0^{\al_n} \dfrac{L_n'/(L_n'+\ep^2)}{\sqrt{\log(\frac{L_n'}{\ep^2} + 1)}}  d\ep \nonumber\\
&= \al_n \sqrt{\log \lp \dfrac{L_n'}{\al_n^2} + 1 \rp} + \sqrt{L_n'} \int_{\sqrt{\log(\frac{L_n'}{\al_n^2} + 1)}}^{\infty} \dfrac{1}{\sqrt{e^{t^2} -1}} dt,
\end{align}
 where the last step is obtained by a change of variable ($\sqrt{\log(L_n'/\ep^2 + 1)} \mapsto t$). It is readily checked that 
$$
\dfrac{1}{\sqrt{e^{t^2} - 1}} \leq e^{-t^2/2} \sqrt{1+\dfrac{\al_n^2}{L_n'}} ~\text{ for}~t \geq \sqrt{\log\lp \dfrac{L_n'}{\al_n^2} + 1 \rp}.
$$  
Consequently we obtain from \eqref{eq:int-sqrtlog-ub-2}:
\beq\label{eq:int-sqrt-log-ub-3}
\int_0^{\al_n} \sqrt{\log \lp \dfrac{L_n'}{\ep^2} + 1 \rp} d\ep \leq \al_n \sqrt{\log \lp \dfrac{L_n'}{\al_n^2} + 1 \rp} + \sqrt{L_n'+\al_n^2} \int_{\sqrt{\log(L_n'/\al_n^2 + 1)}}^{\infty} e^{-t^2/2} dt.
\eeq
Now, \eqref{eq:int-sqrt-log-ub-3} can be represented using the standard normal distribution function as follows:
\beq\label{eq:int-sqrtlog-1}
\int_0^{\al_n} \sqrt{\log \lp \dfrac{L_n'}{\ep^2} + 1 \rp} d\ep = \al_n \sqrt{\log\lp \dfrac{L_n'}{\al_n^2} + 1 \rp} + \sqrt{ 2\pi \lp L_n'+\al_n^2 \rp} \lp 1-\Phi \lp \sqrt{\log \dfrac{L_n'}{\al_n^2} + 1} \rp \rp.
\eeq
Thus, from \eqref{eq:cond-E} and \eqref{eq:int-sqrtlog-1} we obtain:
\begin{multline}\label{eq:Borel-TIS-3}
\E \lc \sup_{t \in [0,L_n]} \hat{f}_t \rc \leq M \lp \al_n \sqrt{\log \lp \dfrac{L_n'}{\al_n^2} + 1 \rp} + \sqrt{2 \pi \lp L_n'+\al_n^2 \rp} \lp 1-\Phi \lp \sqrt{\log \lp \dfrac{L_n'}{\al_n^2}+1 \rp} \rp  \rp \right.\\ + \left.C \al_n \sqrt{\log n}.\right.
\end{multline}
Having completed the first step in our attempt to bound $\ga_n$, we now proceed to obtain an upper bound to the right hand side of \eqref{eq:Borel-TIS-3}. The expectation of the first term:
\begin{align*}
\bex\lc \al_n \sqrt{\log(L_n'/\al_n^2 + 1)} \rc &= \int_0^{\infty} \bp \lp \al_n \sqrt{\log\lp\dfrac{L_n'}{\al_n^2} + 1\rp} \geq x \rp dx\\ 
&= \int_{0}^{s_g} \bp \lp \al_n \geq {g}^{-1}(x) \rp dx, 
\end{align*}
where ${g}$ is the function $\tilde{g}(y)=y \sqrt{\log(L_n'/y^2 + 1)}$ restricted to the domain $[0, s_g]$ and $s_g$ is the point of global maxima of $\tilde{g}$, thereby making ${g}$ invertible. In addition it is readily checked by comparing values of $\tilde{g}'$ that $s_g < \sqrt{L_n'/(e-1)}$. Following the above steps we thus obtain:
\begin{align}\label{eq:UB-1}
\bex& \lc \al_n \sqrt{\log (L_n'/\al_n^2 + 1)} \rc = \int_0^{s_g} \bp \lp \al_n \geq t \rp g'(t) dt 
= \int_0^{s_g} \bp \lp \al_n \geq t \rp \dfrac{\log(\frac{L_n'}{t^2} + 1)- \frac{L_n'}{L_n'+t^2}}{\sqrt{\log(\frac{L_n'}{t^2} + 1)}}  \nonumber \\
&\leq \int_0^{k_0 \frac{\log n}{n}} \sqrt{\log\lp\frac{L_n'}{t^2} + 1\rp} dt  + \int_{k_0 \frac{\log n}{n}}^{s_g} \bp \lp \al_n \geq t \rp \sqrt{\log\lp\frac{L_n'}{t^2} + 1\rp} dt. 
\end{align}
The first term on the right hand side in \eqref{eq:UB-1} can be bounded above using \eqref{eq:int-sqrtlog-1} as follows:
\begin{multline*}
 \int_0^{k_0 \frac{\log n}{n}} \sqrt{\log \lp \frac{L_n'}{t^2} + 1 \rp} dt \\
 \leq  k_0 \dfrac{\log n}{n} \sqrt{\log\lp \dfrac{n^2 L_n'}{k^2 (\log n)^2} + 1 \rp} + \sqrt{ 2\pi \lp L_n'+k^2 \frac{(\log n)^2}{n^2} \rp} \lp 1-\Phi \lp \sqrt{\log \lp \dfrac{n^2 L_n'}{k^2 (\log n)^2} + 1\rp} \rp \rp.
\end{multline*}
It is readily checked by using the standard upper bound for normal tail probability that there exists a constant $C'$ such that the right hand side is bounded above by $C_1' \frac{\sqrt{\log(L_n' \vee n)}}{n^{\ga/4}}$ and thus we have:
\beq\label{eq:UB-2}
\int_0^{k_0 \frac{\log n}{n}} \sqrt{\log \lp \frac{L_n'}{t^2} + 1 \rp} dt  \leq C_1' \frac{\sqrt{\log(L_n' \vee n)}}{n^{\ga/4}}.
\eeq
In order to bound the second term in \eqref{eq:UB-1}, let us perform a change of variable manipulation, namely replace $n^{\ga}t^4$ by $e^{-z}$. We now obtain from Assumption~\ref{assum:tBM-3}:
\begin{align*}
&\int_{k_0 \frac{\log n}{n}}^{s_g} \bp \lp \al_n \geq t \rp  \sqrt{\log\lp\frac{L_n'}{t^2} + 1\rp} dt = \int_{0}^{s_g - k_0 \frac{\log n}{n}} k_1 e^{-k_2 4 n^{\ga} t^4 \wedge 2 k_3 n^{\ga} t^2} \sqrt{\log \lp \dfrac{L_n'}{(t+k_0\frac{\log n}{n})^2} + 1\rp}dt\\ 
&\leq \int_{0}^{s_g - k_0 \frac{\log n}{n}} k_1 e^{-k_2 4 n^{\ga} t^4 \wedge 2 k_3 n^{\ga} t^2} \sqrt{\log \lp \dfrac{L_n'}{t^2} + 1\rp}dt \\
&\leq \dfrac{k_1}{4 n^{\ga/4}} \int_{-\infty}^{\infty} \exp{\lp -4k_2e^{-z}\wedge 2k_3 n^{\ga/2} e^{-z/2} -\dfrac{z}{4} \rp} {\sqrt{\log\lp n^{\ga/2} L_n' e^{z/2} + 1 \rp}}dz\\
&\leq \dfrac{k_1}{4 n^{\ga/4}} \int_{-\infty}^{\infty} \exp{\lp -4k_2e^{-z}\wedge 2k_3 n^{\ga/2} e^{-z/2} -\dfrac{z}{4} \rp} {\sqrt{\dfrac{\ga}{2}\log n + \log\lp L_n' e^{z/2} + 1 \rp}}dz.
\end{align*}
Consequently we have:
\begin{multline}
 \int_{k_0 \frac{\log n}{n}}^{s_g} \bp \lp \al_n \geq t \rp \sqrt{\log\lp\frac{L_n'}{t^2} + 1\rp} dt \\ 
\leq \dfrac{k_1 \sqrt{\log (L_n' \vee n)}}{4 n^{\ga/4}} \int_{-\infty}^{\infty} \exp{\lp -4k_2e^{-z}\wedge 2k_3 n^{\ga/2} e^{-z/2} -\dfrac{z}{4} \rp} \sqrt{\frac{\ga}{2} + \frac{\log(L_n' e^{z/2} + 1)}{\log (L_n' \vee n)}} dz.
\end{multline}
It is readily checked that the integral on the right is finite and thus we have:
\beq\label{eq:UB-4}
\int_{k_0 \frac{\log n}{n}}^{s_g} \bp \lp \al_n \geq t \rp \sqrt{\log\lp\frac{L_n'}{t^2} + 1\rp} dt 
 \leq C_2' \dfrac{\sqrt{\log (L_n' \vee n)}}{n^{\ga/4}},
\eeq
for some generic constant $C_2'$. Using \eqref{eq:UB-2} and \eqref{eq:UB-4} in \eqref{eq:UB-1} we now obtain:
\beq\label{eq:UB-5}
\bex\lc \al_n \sqrt{\log (L_n/\al_n^2 + 1)} \rc \leq C_3' \lp\dfrac{\sqrt{\log(L_n' \vee n)}}{n^{\ga / 4}}\rp,
\eeq
where $C_3' = C_1'+C_2'$. For the second term in \eqref{eq:Borel-TIS-3}, we use the bound on the normal tail probability, namely, $1-\Phi(t) \leq \frac{e^{-t^2/2}}{t \sqrt{2 \pi}}$. Thus, we have
$$
\sqrt{2 \pi \lp L_n'+\al_n^2 \rp} \lp 1-\Phi \lp \sqrt{\log\lp  \dfrac{L_n'}{\al_n^2}+ 1 \rp} \rp  \rp \leq \dfrac{\al_n}{\sqrt{\log ({L_n'}+ {\al_n^2})}}.
$$
The right hand side can be bounded above by using the fact that $\frac{1}{\sqrt{\log(L_n' + \al_n^2)}} \leq \frac{1}{\sqrt{\log T_1'}}$, and the bound for $\bex[\al_n]$ achieved in Lemma~\ref{lem:al_n}. Consequently, combining \eqref{eq:E-al_n}, \eqref{eq:Borel-TIS-3} and \eqref{eq:UB-5} we have thus obtained
$$
\bex \lc \sup_{t \in [0,L_n]}  \hat{f_t}  \rc \leq C \lp\dfrac{\sqrt{\log \lp \lp c_{\xi_n}L_n\rp \vee n \rp}}{n^{\ga / 4}}\rp,$$
for some constant $C$.
\end{proof}
We finally arrive at the main result of this section, namely, in Proposition~\ref{prop:3} we state a general non-asymptotic probabilistic bound on the difference between a time-changed Brownian motion evaluated on a stochastic jump process and its fluid limit.
\begin{proposition}\label{prop:3}
Let Assumptions~\ref{assum:tBM-1}, \ref{assum:tBM-2}, \ref{assum:tBM-3} and \ref{assum:tBM-4} hold. Then there exist a constant $C$ such that for all $n \geq 1$ and $x > 0$:
\beq\label{eq:tBM-main}
\bp \lp \sup_{s \in [0,L_n]} \lln B_{\Xi_n(s)} - B_{\xi_n(s)} \rrn > C\dfrac{\sqrt{\log\lp \lp c_{\xi_n}L_n \rp \vee n\rp}}{n^{1/4}} + x \rp \leq 2 e^{-\frac{x^2 }{2 \ups_n^2}},
\eeq
where $\ups_n^2 = \sup_{s \in [0,L_n]} \bex[ |\Xi_n(s) - \xi_n(s)| ]$.
\end{proposition}

\begin{proof}Observe that we have:
\begin{align}\label{eq:E-f_t^2}
\E \lc \hat{f_t}^2 \rc &= \E \lc B_{\xi_n(t)}^2 + B_{\Xi_n(t)}^2 - 2 B_{\xi_n(t)} B_{\Xi_n(t)} \rc \nonumber\\
&= \lc \xi_n(t)+\Xi_n(t) - 2 \xi_n(t) \wedge \Xi_n(t) \rc \nonumber \\
&= \lc |\Xi_n(t) - \xi_n(t)| \rc.
\end{align} 
This implies
\beq\label{eq:si^2-ub}
\sup_{t \in [0,L_n]} \bex \lc  \hat{f}_t^2 \rc = \sup_{t \in [0,L_n]} \bex[|\Xi_n(t) - \xi_n(t)|] = \ups_n^2. 
\eeq
Recall $\ga_n := \bex \lc  \sup_{t \in [0,1]} \hat{f_t}  \rc$. Then by the Borell-TIS inequality (see e.g. \cite[Theorem 2.1.1]{adler}) we have:
\beq\label{eq:Borell-TIS}
\bp  \lp \sup_{t \in [0, L_n]} \lln \hat{f_t} \rrn > x + \ga_n  \rp \leq 2 \exp \lp \dfrac{-x^2}{2 \ups_n^2} \rp.
\eeq
Now invoking Lemma~\ref{lem:ga_n} we obtain our desired result \eqref{eq:tBM-main}.
\end{proof}
{
Finally, we will require the following proposition for proving strong embeddings under Assumption~\ref{assum:arriv-dropouts-2}. Consider the following regularity condition.

\begin{assumption}\label{assum:tBM-nonstoch}
For every $n \geq 1$, let $\xi^{(n)}, \xi:[0,T]\mapsto \R$ satisfy:
\beq
\sup_{s \in [0,T]} \lln \xi^{(n)}(s) - \xi(s) \rrn = O \lp \dfrac{1}{\sqrt{n}} \rp.
\eeq
In addition, let $\xi^{(n)}$ and $\xi$ both be Lipschitz continuous with the Lipschitz coefficient of $\xi^{(n)}$ growing at most polynomially in $n$. 
\end{assumption}
\begin{proposition}\label{prop:tBM-nonstoch}
Let Assumption~\ref{assum:tBM-nonstoch} hold and $B^n, n\geq 1$ be any sequence of Brownian motions defined on a probability space $({\Omega}, {\cf}, \bp)$. Then there exists constants $C$, $K$ and $\la$ such that:
$$
\bp \lp \sup_{s \in [0,T]} \lln B_{\xi^{(n)}(s)} - B_{\xi(s)} \rrn > C \dfrac{\sqrt{\log n}}{n^{1/4}} + x \rp < K e^{-\la x^2 \sqrt{n}}.
$$
\end{proposition}
\begin{proof}
The key ingredient of the proof is again the Borel-TIS inequality. First, let us reuse the same notations as before;  viz., let $\hat{f_s}$ be the Gaussian process defined on $[0,T]$ by 
$$
\hat{f_s} = B_{\xi^{(n)}(s)} - B_{\xi(s)}.
$$
As before, the canonical metric for $\hat{f}$ in $({\Omega}, {\cf}, \bp)$ is given by \eqref{eq:can-metr-1} and let $\hat{D}$ denote the diameter of $[0,T]$ with respect to the canonical metric, i.e.,
$$
\hat{D} = \sup_{s,t \in [0,T]} \hat{d}(s,t).
$$
Let $\hat{N}(\ep)$ be the metric entropy defined by the smallest number of balls of diameter $\ep$ (with respect to the canonical metric $\hat{d}$) that cover $[0,T]$. Then from \cite[Theorem 1.3.3]{adler} there exists a universal constant $M$ such that 
\beq\label{eq:E-sup-ub-2}
\bex \lc \sup_{t \in [0,T]} \hat{f_t} \rc \leq M \int_0^{\hat{D}/2} {\lp \log \hat{N}(\ep) \rp}^{1/2} d\ep.
\eeq
It can be easily shown that the canonical metric $\hat{d}$ as defined in \eqref{eq:can-metr-1} satisfies:
\begin{multline*}
\hat{d}(s,t)^2 =  \lc \lln \xi^{(n)}(s)-\xi(s) \rrn + \lln \xi^{(n)}(t)-\xi(t) \rrn -2 (\xi^{(n)}(s) \wedge \xi(t) + \xi^{(n)}(s) \wedge \xi(t) \right.\\
 \left.- \xi^{(n)}(s) \wedge \xi(t) -  \xi^{(n)}(t) \wedge \xi(s)) \rc.
\end{multline*}
If the numbers $\xi(s),\xi^{(n)}(s),\xi(t),\xi^{(n)}(t)$ be arranged in ascending order form the vector $(d_1,d_2,d_3,d_4)$, then it can be shown that 
\beq\label{eq:dhat-2}
\hat{d}(s,t) = \sqrt{(d_4-d_3) + (d_2 - d_1)}.
\eeq
This implies that we have
$$
\hat{d}(s,t) \leq \sqrt{|\xi^{(n)}(t)-\xi(t)|+ |\xi^{(n)}(s) - \xi(s)|}.
$$
Consequently from Assumption~\ref{assum:tBM-nonstoch} we have
\beq\label{eq:diam-O}
\hat{D} = O \lp \dfrac{1}{n^{1/4}} \rp.
\eeq
In addition, \eqref{eq:dhat-2} also implies:
$$
\hat{d}(s,t) \leq \sqrt{|\xi^{(n)}(t)-\xi^{(n)}(s)|+ |\xi(t)-\xi(s)|}.
$$
Using the Lipschitz continuity of $\xi^{(n)}$ and $\xi$, we obtain that
\beq\label{eq:dhat-2-2}
\sup_{s, t \in [0,T]} \dfrac{\hat{d}(s,t)}{\sqrt{|s-t|}} = l_n,
\eeq
where $l_n$ grows polynomially in $n$. This implies $\hat{d}(s,t) \leq l_n \sqrt{|s-t|}$ for all $s,t \in [0,T]$. Thus $\{s \in [0,T]:|s-x_0| \leq \frac{\ep^2}{l_n^2} \}$ is contained in the $\ep-$ball around $x_0$. The length of this ball is thus at least $\frac{2 \ep^2}{l_n^2}$. Hence the number of $\ep-$balls which cover $[0,T]$ is atmost $\frac{T l_n^2}{2 \ep^2}$. This leads to an upper bound for $\hat{N}(\ep)$, namely
\beq\label{eq:N(ep)-ub-2}
\hat{N}(\ep) \leq \dfrac{T l_n^2}{2 \ep^2}.
\eeq 
Thus, using \eqref{eq:N(ep)-ub-2} we get from \eqref{eq:E-sup-ub-2}:
\beq\label{eq:Borell-TIS-2-2}
\bex \lc \sup_{t \in [0,T]} \hat{f}_t \rc \leq M \int_0^{\hat{D}/2} \sqrt{\log \lp \dfrac{T l_n^2}{\ep^2} \rp} d\ep.
\eeq
It is readily checked using integration by parts that
$$
\bex \lc \sup_{t \in [0,T]} \hat{f}_t \rc \leq M \lp \dfrac{\hat{D}}{2} \sqrt{\log \lp \dfrac{4 T l_n^2}{\hat{D}^2}\rp} + \int_0^{\hat{D}/2} \dfrac{1}{\sqrt{\log \frac{T l_n^2}{\ep^2}}} d\ep \rp .
$$
A change of variable $\sqrt{\log \frac{T l_n^2}{\ep^2}} \mapsto t$ in the integral on the right hand side yields
$$
\bex \lc \sup_{t \in [0,T]} \hat{f}_t \rc \leq M \lp \dfrac{\hat{D}}{2} \sqrt{\log \lp \dfrac{4T l_n^2}{\hat{D}^2}\rp} + \sqrt{2\pi T} l_n \lp 1- \Phi\lp \sqrt{2 \log \lp \dfrac{2l_n\sqrt{T}}{\hat{D}} \rp} \rp \rp \rp.
$$
The standard upper bound to the normal tail probability now gives:
\beq\label{eq:E-sup-f-ub-1}
\bex \lc \sup_{t \in [0,T]} \hat{f}_t \rc \leq M \lp \dfrac{\hat{D}}{2} \sqrt{\log \lp \dfrac{4 T l_n^2}{\hat{D}^2}\rp} + \dfrac{\hat{D}}{2} \dfrac{1}{\sqrt{2 \log \frac{2 l_n \sqrt{T}}{\hat{D}}}}\rp.
\eeq
Observe that since $l_n \geq 1$ and $\hat{D} \leq 2$ we have:
\beq\label{eq:ub-A}
\dfrac{1}{\sqrt{2 \log \frac{2 l_n \sqrt{T}}{\hat{D}}}} \leq \dfrac{1}{\sqrt{\log T}}
\eeq
In addition, since $x\log x \geq -1$ for all $x>0$:
\beq\label{eq:ub-B}
\dfrac{\hat{D}}{2} \sqrt{\log \lp \dfrac{4Tl_n^2}{\hat{D}^2} \rp} = \dfrac{1}{2} \sqrt{\hat{D}^2 \log(4Tl_n^2) - \hat{D}^2 \log \hat{D^2}} \leq \dfrac{1}{2} \sqrt{\hat{D}^2 \log(4Tl_n^2) + 1}.
\eeq
Using the inequalities \eqref{eq:ub-A} and \eqref{eq:ub-B} in the right hand side of \eqref{eq:E-sup-f-ub-1} we get:
\beq\label{eq:E-sup-f-ub-2}
\bex \lc \sup_{t \in [0,T]} \hat{f}_t \rc \leq M \lp \dfrac{1}{2} \sqrt{\hat{D}^2 \log(4Tl_n^2)+1} + \dfrac{\hat{D}}{2} \dfrac{1}{\sqrt{\log T}}\rp.
\eeq
Finally using \eqref{eq:diam-O} and \eqref{eq:dhat-2-2} in \eqref{eq:E-sup-f-ub-2} we obtain:
$$
\bex \lc \sup_{t \in [0,T]} \hat{f}_t \rc = O \lp \dfrac{\sqrt{\log n}}{n^{1/4}} \rp.
$$
Observe that 
$$
\sup_{s \in [0,T]} \bex \lc |f_t^2| \rc = \sup_{s \in [0,T]} \lln \xi^{(n)}(s)-\xi(s) \rrn = O\lp \dfrac{1}{\sqrt{n}} \rp.
$$
Then by the Borell-TIS inequality (see e.g. \cite[Theorem 2.1.1]{adler}) there exists constants $C$, $K$ and $\la$ such that:
$$
\bp \lp \sup_{s \in [0,1]} \lln B_{\xi^{(n)}(s)} - B_{\xi(s)} \rrn > C \dfrac{\sqrt{\log n}}{n^{1/4}} + x \rp < K e^{-\la x^2 \sqrt{n}}.
$$
\end{proof}
}
\section{A Strong Embedding for the Arrival Process}\label{sec:arrive}
In this section we derive a strong embedding for the arrival process. The following proposition is an extension of Theorem~\ref{thm:KMT-part-sum} when the length of the random walk is provided by a time-varying not necessarily determinstic function.
\begin{proposition}\label{prop:1}
Let $X_1, \ldots, X_n$ be iid samples from a distribution which admits existence of a moment generating function in a neighborhood of zero. Let $\mu$ and $\si$ denote the mean and standard deviation respectively of this distrbution. Let $J_n :[0, \infty) \mapsto \{1,\ldots,n\}$ be any process. Then there exists a standard Brownian motion $B$ and a version of $X_1, \ldots, X_n$, along with constants $C_1$, $K_1$ and $\la_1$ such that for all $x>0$ we have:
$$
\bp \lc \sup_{t \in [0,\infty)}  \lln \dfrac{1}{\sqrt{n}} \sum_{i=1}^{J_n(t)} \lp X_i - \mu \rp - \si B_{\frac{J_n(t)}{n}} \rrn > C_1 \dfrac{\log n}{\sqrt{n}} + x\rc < K_1 e^{-\la_1 x \sqrt{n}}.
$$
\end{proposition}
\begin{proof}
From Theorem~\ref{thm:KMT-part-sum} we have that there exists a standard Brownian motion $\tilde{B}$, {a version of $X_1, \ldots, X_n$}, along with constants $C$, $K$ and $\la$ (depending on the distribution of $V$) such that for all $x>0$ we have:
\beq\label{eq:2-a}
\bp \lc \sup_{0 \leq k \leq n} \lln \sum_{i=1}^k \lp \dfrac{X_i - \mu}{\si} \rp - \tilde{B}_k \rrn > C {\log n} + x \rc < K e^{-\la x},
\eeq
where $\sum_{i=1}^k$ is defined to be the null sum for $k=0$. Since $A_n$ takes values in $\{1,\ldots, n\}$, we may replace the supremum in the left hand side of \eqref{eq:2-a} by a supremum over $k$ taking values in $\lbrace A_n(t), t \in [0,\infty) \rbrace$. Consider another version $B$ of the standard Brownian motion $\tilde{B}$ such that 
$$
\sqrt{n}B_{k/n} \stackrel{d}{=} \tilde{B}_{k}
$$
Then from \eqref{eq:2-a} there exist constants $C_1$, $K_1$ and $\la_1$ such that for all $x>0$ the desired strong-embedding holds:
$$
\bp \lc \sup_{t \in [0,\infty)}  \lln \dfrac{1}{\sqrt{n}} \sum_{i=1}^{J_n(t)} \lp {X_i - \mu}\rp - \si B_{\frac{A_n(t)}{n}}  \rrn > C_1 \dfrac{\log n}{\sqrt{n}} + x \rc < K_1 e^{-\la_1 x \sqrt{n}}.
$$
\end{proof}

The following proposition is a consequence of Theorem~\ref{thm:KMT-emp-pro} and holds for any general distribution as opposed to the uniform distributional assumption made in Theorem~\ref{thm:KMT-emp-pro}. 
\begin{proposition}\label{prop:2}
Let Assumption~\ref{assum:arriv-dropouts} hold with $p=1$, that is, let us consider the $\Delta_{(i)}/G/1$ model as explained in Remark~\ref{rem:Delta_i}. Then for every $n \geq 1$ there exists a Brownian bridge $\lbrace{B^{\textnormal{br},n}};t \in [0,1]\rbrace$ {and a version of $T_1, \ldots, T_n$} along with constants $C_2$, $K_2$ and $\la_2$ such that for all $x>0$ we have:
\beq\label{eq:2-k}
\bp \lp \sup_{t \in [0, \infty)} \lln \sqrt{n} \lp G_n(t) - G(t) \rp - B_{G(t)}^{\textnormal{br},n} \rrn > C_2 \dfrac{\log n}{\sqrt{n}} + x \rp < K_2 e^{-\la_2 x \sqrt{n}}.
\eeq
\end{proposition}

\begin{proof}
We will first consider the random variables $\{ G(T_i) : i=1,\ldots,n \}$. Observe that the $G(T_i)$'s are independent and identically distribued as $U[0,1]$ random variables. Consider the corresponding empirical distribution function $F_n$ given by:
\beq\label{eq:2-b}
F_n(t) = \dfrac{1}{n} \sum_{i=1}^n \mathbf{1}_{\{ G(T_i) \leq t \}}.
\eeq
Then by a little modification of \eqref{eq:strong-approx-1} there exist a Brownian bridge $B^{\textnormal{br}, n}$ (observe by Remark~\ref{rem:bridge-diff-n}, the Brownian bridge under consideration depends on $n$), constants $C_2$, $K_2$ and $\la_2$ such that:
\beq\label{eq:2-e}
\bp \lp \sup_{t \in [0,1]} \lln \sqrt{n} \lp F_n(t)-t \rp - B_t^{\textnormal{br},n} \rrn > C_2 \dfrac{\log n}{\sqrt{n}} + x\rp < K_2 e^{-\la_2 x \sqrt{n}}.
\eeq
Let the inverse distribution function $G^{-1}$ be defined as:
\beq\label{eq:2-1}
G^{-1} (t) := \sup \lcl x \in \R: G(x) \leq t \rcl.
\eeq
Observe that owing to our definition of $G^{-1}$ we have:
\beq\label{eq:2-c}
G(x) \leq t ~~\textnormal{iff}~~x \leq G^{-1}(t).
\eeq
Applying relation \eqref{eq:2-c} in \eqref{eq:2-b}, and using \eqref{eq:2-d} we obtain:
\beq\label{eq:2-f}
F_n(t) = \dfrac{1}{n} \sum_{i=1}^n \mathbf{1}_{\{ T_i \leq G^{-1}(t)\}} = G_n (G^{-1}(t)).
\eeq
Inserting \eqref{eq:2-f} in \eqref{eq:2-e} yields:
\beq\label{eq:2-g}
\bp \lp \sup_{t \in [0,1]} \lln \sqrt{n} \lp G_n (G^{-1}(t)) - t \rp - B_t^{\textnormal{br},n} \rrn > C_2 \dfrac{\log n}{\sqrt{n}} + x \rp < K_2 e^{-\la_2 x \sqrt{n}}.
\eeq
In addition, observe that for any $s_1 < s_2$ such that $G(s_1) = G(s_2)$ we have for all $i = 1, \ldots, n$:
$$
P \lc T_i \in (s_1, s_2] \rc = 0.
$$
This implies even though $G^{-1}(G(s)) \geq s$, we still have:
$$
\mathbf{1}_{\{T_i \leq G^{-1}(G(s))\}} = \mathbf{1}_{\{T_i \leq s\}} ~ \textnormal{a.s.}
$$
Consequently we obtain:
\beq\label{eq:2-h}
G_n(G^{-1}(G(s))) = \dfrac{1}{n} \sum_{i=1}^n \mathbf{1}_{\{ T_i \leq G^{-1}(G(s)) \}} = \dfrac{1}{n} \sum_{i=1}^n \mathbf{1}_{\{T_i \leq s\}} = G_n(s) ~\textnormal{a.s.}
\eeq
We utilize this property in \eqref{eq:2-g}. Notice that 
$$
\lcl G(s) : s \in [0, \infty) \rcl \subset [0,1].
$$
Thus we have the following inequality between the suprema of the same function over these two sets:
\beq\label{eq:2-i}
\sup_{t \in [0,1]} \lln \sqrt{n} \lp G_n(G^{-1}(t)) - t \rp - B_t^{\textnormal{br},n} \rrn \geq \sup_{s \in [0, \infty)} \lln \sqrt{n} \lp G_n(G^{-1}(G(s))) - G(s) \rp - B_{G(s)}^{\textnormal{br},n} \rrn. 
\eeq
Inserting \eqref{eq:2-h} in \eqref{eq:2-i} we thus get
\beq\label{eq:2-j}
\sup_{t \in [0,1]} \lln \sqrt{n} \lp G_n (G^{-1}(t)) - t \rp - B_t^{\textnormal{br},n} \rrn \geq \sup_{s \in [0, \infty)} \lln \sqrt{n} \lp G_n(s) - G(s) \rp - B_{G(s)}^{\textnormal{br},n} \rrn~\textnormal{a.s.}
\eeq
Looking at the complement probability in the left hand side of \eqref{eq:2-g} we obtain as a result of \eqref{eq:2-j}:
\begin{align*}
1-K_2e^{-\la_2 x \sqrt{n}} &< \bp \lp \sup_{t \in [0,1]} \lln \sqrt{n} \lp G_n(G^{-1}(t))-t \rp - B_{t}^{\textnormal{br},n} \rrn \leq C_2 \dfrac{\log n}{\sqrt{n}} + x \rp \\
&\leq \bp \lp \sup_{s \in [0,\infty)} \lln \sqrt{n}  \lp G_n(s) - G(s) \rp - B_{G(s)}^{\textnormal{br},n}\rrn \leq C_2 \dfrac{\log n}{\sqrt{n}} + x \rp.
\end{align*}
This yields our desired result \eqref{eq:2-k}.
\end{proof}

{
\begin{remark}\label{rem:const-str-emp}
Observe that the constants $C_2$, $K_2$ and $\la_2$ in \eqref{eq:2-e} do not depend on $G$ and that the same constants satisfy \eqref{eq:2-k}. Thus owing to Remark~\ref{rem:KMT-emp-pro-const} we have that $C_2 =100$, $K_2 = 10$ and $\la_2 = 1/50$ satisfy \eqref{eq:2-k}.
\end{remark}
We now adapt the statement of Proposition~\ref{prop:2} under Assumption~\ref{assum:arriv-dropouts-2}.
\begin{corollary}\label{cor:prop-2-cor}
Let Assumption~\ref{assum:arriv-dropouts-2} hold. Then for every $n \geq 1$ there exists a Brownian bridge $\lbrace{\tilde{B}^{\textnormal{br},n}};t \in [0,1]\rbrace$ {and a version of $T_1, \ldots, T_n$} such that for all $x>0$, the same constants $C_2$, $K_2$ and $\la_2$ as in Proposition~\ref{prop:2} satisfy:
$$
\bp \lp \sup_{t \in [0, \infty)} \lln \sqrt{n}\lp G_n^{(n)}(t) - G^{(n)}(t) - \tilde{B}_{G^{(n)}(t)}^{\textnormal{br},n} \rp \rrn > C_2 \dfrac{\log n}{\sqrt{n}} + x \rp < K_2 e^{-\la_2 x \sqrt{n}}.
$$
\end{corollary}
\begin{proof}
Observe from Remark~\ref{rem:const-str-emp} for every $k \geq 1$, there exists a Brownian bridge $B^{\textnormal{br},k,n}$ such that for all $x>0$, the same constants $C_2$, $K_2$ and $\la_2$ as in Proposition~\ref{prop:2} satisfy:
$$
\bp \lp \sup_{t \in [0, \infty)} \lln \sqrt{n} \lp G_n^{(k)}(t) - G^{(k)}(t) \rp - B_{G^{(k)}(t)}^{\textnormal{br},k,n} \rrn > C_2 \dfrac{\log n}{\sqrt{n}} + x \rp < K_2 e^{-\la_2 x \sqrt{n}}.
$$
In particular, for $k=n$ and writing $B^{\textnormal{br},n,n}$ as $\tilde{B}_{G^{(n)}(t)}^{\textnormal{br},n}$ we have:
$$
\bp \lp \sup_{t \in [0, \infty)} \lln \sqrt{n} \lp G_n^{(n)}(t) - G^{(n)}(t) \rp - \tilde{B}_{G^{(n)}(t)}^{\textnormal{br},n} \rrn > C_2 \dfrac{\log n}{\sqrt{n}} + x \rp < K_2 e^{-\la_2 x \sqrt{n}}.
$$
\end{proof}}

In the sequel, we will require control over Brownian motion evaluated at the fluid-scaled arrival process $A_n/n$ and the corresponding fluid limit. This is achieved for $U[0,1]$ distributed time epochs in the following proposition.
\begin{proposition}\label{prop:BM-dropouts}
Let $T_1, \ldots, T_n$ be iid samples from the $U[0,1]$ distribution. Let $\hat{A}_n$ be the arrival process with dropouts given by
$$
\hat{A}_n(t) = \sum_{i=1}^{nF_n(t)} \zeta_i,
$$
where $F_n$ is the empirical distribution function corresponding to the sample $T_1, \ldots T_n$, and $\zeta_i$ are iid $\textnormal{Ber}(p)$. Let $B$ be a Brownian motion independent of $T_1, \ldots, T_n$ and $\zeta_1, \ldots, \zeta_n$. for each $n \geq 1$. Then, there exist constants $C_3$, $K_3$ and $\la_3$ such that
$$
\bp \lp \sup_{s \in [0,1]} \lln  B_{\frac{\hat{A}_n(s)}{n}} - B_{ps} \rrn > C_3 \dfrac{\sqrt{\log n}}{n^{1/4}}  + x \rp \leq K_3 e^{-\la_3 x^2 \sqrt{n}}
$$
\end{proposition}

\begin{proof}
The proof follows from the DKW-type inequality established for the Brownian motion in Proposition~\ref{prop:3}, the conditions for which are satisfied in Proposition~\ref{lem:tBM-satisfy-1}. Consequently we obtain there exists constants $C_3$, $K_3$ and $\la_3$ such that:
\beq\label{eq:3-a-}
\bp \lp \sup_{s \in [0,1]} \lln B_{\frac{\hat{A}_n(s)}{n}} - B_{ps} \rrn > C_3 \dfrac{\sqrt{\log n}}{n^{1/4}} + x \rp \leq 2 e^{-\frac{x^2}{\si_n^2}},
\eeq
where 
\beq\label{eq:3-a}
\si_n^2 = \sup_{s \in [0,1]} \bex\lp \lln \frac{\hat{A}_n(s)}{n}-ps \rrn \rp.
\eeq
In order to bound $\si_n^2$, we first apply Cauchy-Schwarz inequality to get:
\beq\label{eq:3-b}
\bex\lp \lln \dfrac{\hat{A}_n(s)}{n} - ps \rrn \rp \leq \sqrt{\bex{\lp \dfrac{\hat{A}_n(s)}{n} - ps \rp}^2}.
\eeq
Observe $\hat{A}_n(s)$ has the ${\rm Bin} (n, ps)$ distribution, which implies:
\beq\label{eq:3-c}
\bex{\lp \dfrac{\hat{A}_n(s)}{n} - ps \rp}^2 = \dfrac{\textnormal{Var}(\hat{A}_n(s))}{n^2}= \dfrac{ps(1-ps)}{n}.
\eeq
Combining \eqref{eq:3-b} and \eqref{eq:3-c} we obtain from \eqref{eq:3-a}:
\beq\label{eq:3-d}
\si_n^2 = \sup_{s \in [0,1]} \bex\lp \lln \dfrac{\hat{A}_n(s)}{n} - ps \rrn \rp \leq \sup_{s \in [0,1]} \sqrt{\dfrac{ps(1-ps)}{n}}= \dfrac{1}{2 \sqrt{n}}.
\eeq
Using inequality \eqref{eq:3-d} in \eqref{eq:3-a-} we now have our desired strong embedding result:
$$
\bp \lp \sup_{s \in [0,1]} \lln B_{\frac{\hat{A}_n(s)}{n}} - B_{ps} \rrn > C_3 \dfrac{\sqrt{\log n}}{n^{1/4}} + x \rp \leq 2 e^{-2 x^2 \sqrt{n}},
$$
which holds for all $n \geq 1$ and $x > 0$.
\end{proof}
We now extend our result in Proposition~\ref{prop:BM-dropouts} to generally distributed time epochs. This is the subject of the following corollary.
\begin{corollary}\label{cor:1}
	Let Assumption~\ref{assum:arriv-dropouts} holds. Let $B$ be a Brownian motion independent of $T_1, \ldots, T_n$ and $\zeta_1, \ldots, \zeta_n$ for each $n \geq 1$. Then for every $n \geq 1$ and $x > 0$, the same constants $C_3$, $K_3$ and $\la_3$ as in Proposition~\ref{prop:BM-dropouts} satisfy:
	$$
	\bp \lp \sup_{t \in [0,\infty)} \lln B_{\frac{A_n(t)}{n}} - B_{pG^{(n)}(t)}\rrn > C_3 \dfrac{\sqrt{\log n}}{n^{1/4}} +x \rp \leq K_3 e^{-\la_3 x^2 \sqrt{n}}.
	$$
\end{corollary}
\begin{proof}
The proof of this result is similar to the reasonings we adopted in Proposition~\ref{prop:2}. Consequently, let us again consider the random variables $\{ G(T_i): i=1,\ldots,n \}$. Observe that the $G(T_i)$'s are independent and identically distributed as $U[0,1]$ random variables. Consider the corresponding distribution function $F_n$ given by:
$$
F_n(t) = \dfrac{1}{n} \sum_{i=1}^n \mathbf{1}_{\{ G(T_i) \leq t \}}.
$$ 
Let $\hat{A}_n(t) = \sum_{i=1}^{nF_n(t)} \zeta_i$.
In addition, recall the definition of $G^{-1}(t)$ in \eqref{eq:2-1}. Using \eqref{eq:2-f} we obtain:
\beq\label{eq:3-d-e}
\sup_{s \in [0,1]} \lln B_{ps} - B_{\frac{\hat{A}_n(s)}{n}} \rrn = \sup_{s \in [0,1]} \lln B_{ps} - B_{\frac{A_n({G}^{-1}(s))}{n}} \rrn.
\eeq
In a spirit similar to what is used to obtain \eqref{eq:2-j} we have from the analogue to \eqref{eq:2-h}:
\beq\label{eq:3-e}
\sup_{s \in [0,1]} \lln B_{ps} - B_{\frac{{A}_n({G}^{-1}(s))}{n}} \rrn \geq \sup_{s \in [0, \infty)} \lln B_{\frac{A_n(s)}{n}} - B_{pG(s)} \rrn~ \textnormal{a.s.}
\eeq
Our desired result now follows from Proposition~\ref{prop:BM-dropouts}. Observe Proposition~\ref{prop:BM-dropouts} guarantees existence of constants $C_3$, $K_3$ and $\la_3$ such that, for all $n \geq 1$ and $x > 0$ we have:
\beq\label{eq:3-f}
\bp \lp \sup_{s \in [0,1]} \lln B_{\frac{\hat{A}_n(s)}{n}} - B_{ps} \rrn > C_3 \dfrac{\sqrt{\log n}}{n^{1/4}} + x \rp \leq K_3 e^{-\la_3 x^2 \sqrt{n}}.
\eeq
We now complete our proof by looking at the complement probability in \eqref{eq:3-f} and using \eqref{eq:3-d-e} we have:
\begin{align}\label{eq:3-g}
1-K_3 e^{-\la_3 x^2 \sqrt{n}} &\leq \bp \lp \sup_{s \in [0,1]} \lln B_{\frac{\hat{A}_n(s)}{n}} - B_{ps} \rrn \leq C_3 \dfrac{\sqrt{\log n}}{n^{1/4}} + x \rp \nonumber \\
&= \bp \lp \sup_{s \in [0,1]} \lln B_{\frac{A_n({G}^{-1}(s))}{n}} - B_{ps} \rrn \leq C_3 \dfrac{\sqrt{\log n}}{n^{1/4}} + x \rp
\end{align}
Now using \eqref{eq:3-e} in \eqref{eq:3-g} we obtain:
\beq
1-K_3 e^{-\la_3 x^2 \sqrt{n}} \leq \bp \lp \sup_{s \in [0,\infty)} \lln B_{\frac{A_n(s)}{n}} - B_{pG(s)} \rrn \leq C_3 \dfrac{\sqrt{\log n}}{n^{1/4}} \rp,
\eeq
which yields our desired result.
\end{proof}
\begin{remark}\label{rem:cor-1}
Observe that the constants $C_3$, $K_3$ and $\la_3$ in Corollary~\ref{cor:1} do not depend on the particular distribution $G$.
\end{remark}
Corollaries~\ref{cor:prop-2-cor} and \ref{cor:1} provide approximations in terms of $G^{(n)}$. In order to further simplify our approximation processes we require the following lemma.
{\begin{lemma}\label{lem:B_G^n-G}
Let Assumption~\ref{assum:arriv-dropouts-2} hold. Then for any $q>0$ and any sequence of Brownian motions $B^n$, there exist constants $C_4$, $K_4$ and $\la_4$ such that
$$
\bp \lp \sup_{t \in [0, \infty)} \lln B_{qG^{(n)}(t)}^n - B_{qG(t)}^n \rrn > C_4 \dfrac{\sqrt{\log n}}{n^{1/4}} + x \rp < K_4 e^{-\la_4 x^2 \sqrt{n}}.
$$
\end{lemma}
\begin{proof}
We apply Proposition~\ref{prop:tBM-nonstoch}. Take $T=1$, $\xi^{(n)}(s) = qG^{(n)}(G^{-1}(s))$, $\xi(s)=qs$ and observe that continuity of $G$ implies
\begin{multline*}
\sup_{s \in [0,1]} \lln \xi^{(n)}(s) - \xi(s) \rrn = \sup_{s \in [0,1]} q \lln G^{(n)}(G^{-1}(s)) - s \rrn = \sup_{s \in [0, \infty)} q\lln G^{(n)}(G^{-1}(G(s))) - G(s) \rrn.
\end{multline*}
Notice $G(G^{-1}(G(s)))=G(s)$ for all $s$. Thus we have:
$$
\sup_{s \in [0,1]} \lln \xi^{(n)}(s) - \xi(s) \rrn = \sup_{s \in [0, \infty)} \lln G^{(n)}(G^{-1}(G(s))) - G(G^{-1}(G(s))) \rrn \leq \sup_{s \in [0, \infty)} \lln G^{(n)}(s) - G(s) \rrn.
$$
From \eqref{eq:Gn-dkw} we now obtain that
$$
\sup_{s \in [0,1]} \lln \xi^{(n)}(s) - \xi(s) \rrn = O\lp \dfrac{1}{\sqrt{n}} \rp.
$$
The Lipschitz continuity of $\xi$ is obvious, while Lipschitz continuity of $\xi^{(n)}$ follows from that of $G^{(n)}$ and $G^{-1}$. In addition, from a similar property for $G^{(n)}$, the Lipschitz coefficient of $\xi^{(n)}$ grows at most polynomially in $n$. 
Thus we have constants $C_4$, $K_4$ and $\la_4$ such that:
\beq\label{eq:B-xi^n-xi}
\bp \lp \sup_{t \in [0,1]} \lln B_{qG^{(n)}(G^{-1}(s))} - B_{qs} \rrn > C_4 \dfrac{\sqrt{\log n}}{n^{1/4}} + x \rp < K_4 e^{-\la_4 x^2 \sqrt{n}}.
\eeq
Since $G$ is strictly increasing in $[0, \infty)$, we have $G^{-1}(G(s))=s$ for $s \in [0, \infty)$. Consequently we have:
$$
\sup_{s \in [0,1]} \lln B_{qG^{(n)}(G^{-1}(s))} - B_{qs} \rrn = \sup_{s \in [0, \infty)} \lln B_{qG^{(n)}(G^{-1}(G(s)))} - B_{qG(s)} \rrn = \sup_{s \in [0, \infty)} \lln B_{qG^{(n)}(s)} - B_{qG(s)} \rrn.
$$
This provides our desired result from \eqref{eq:B-xi^n-xi}.
\end{proof}
We can now extend Corollary~\ref{cor:1} courtesy Lemma~\ref{lem:B_G^n-G}.
\begin{corollary}
Suppose Assumption~\ref{assum:arriv-dropouts-2} holds. Let $B$ be a Brownian motion independent of $T_1, \ldots, T_n$ and $\zeta_1,\ldots,\zeta_n$ for each $n \geq 1$. Then there exist constants $C_5$, $K_5$ and $\la_5$ such that for all $n \geq 1$ and $x>0$ we have:
$$
\bp \lp \sup_{t \in [0, \infty)} \lln B_{\frac{A_n(t)}{n}} - B_{pG(t)} \rrn > C_5 \dfrac{\sqrt{\log n}}{n^{1/4}} + x \rp < K_5 e^{-\la_5 x^2 \sqrt{n}}.
$$
\end{corollary}
\begin{proof}
From Corollary~\ref{cor:1} and observing Remark~\ref{rem:cor-1} we have
$$
\bp \lp \sup_{t \in [0, \infty)} \lln B_{\frac{A_n(t)}{n}} - B_{p G^{(n)}(t)} \rrn > C_3 \dfrac{\sqrt{\log n}}{n^{1/4}} + x \rp < K_3 e^{-\la_3 x^2 \sqrt{n}}.
$$
From Lemma~\ref{lem:B_G^n-G} with $q=p$ we have constants $C_4'$, $K_4'$ and $\la_4'$ such that
$$
\bp \lp \sup_{t \in 0, \infty)} \lln B_{pG^{(n)}(t)} - B_{pG(t)} \rrn > C_4' \dfrac{\sqrt{\log n}}{n^{1/4}} + x \rp < K_4' e^{-\la_4' x^2 \sqrt{n}}.
$$
Combining the above we have our desired result.
\end{proof}
Corollary~\ref{cor:prop-2-cor} can be extended using the following result which again is a consequence of Lemma~\ref{lem:B_G^n-G}.
\begin{corollary}\label{cor:B-br_n(G)^n-G}
Let Assumption~\ref{assum:arriv-dropouts-2} hold. Then we have:
$$
\bp \lp \sup_{t \in [0, \infty)} \lln \tilde{B}_{G^{(n)}(t)}^{\textnormal{br},n} - \tilde{B}_{G(t)}^{\textnormal{br},n} \rrn > C_6 \dfrac{\sqrt{\log n}}{n^{1/4}} + x \rp < K_6 e^{-\la_6 x^2\sqrt{n}}.
$$
\end{corollary}
\begin{proof}
Using the fact that the Brownian bridge $\tilde{B}^{\textnormal{br},n}$ can be represented as:
$$
\tilde{B}^{\textnormal{br},n}_t \stackrel{D}{=} B_t^n - t B_1^n,
$$
for a Brownian motion $B^{n}$ we have that
\begin{align}\label{eq:B-br_n(G)^n-G-1}
\bp \left(\sup_{t \in [0, \infty)}\lln \tilde{B}_{G^{(n)}(t)}^{\textnormal{br},n} - \tilde{B}_{G(t)}^{\textnormal{br},n}  \rrn > 2 z\right) \leq &\bp \left(\sup_{t \in [0, \infty)}\lln B_{G^{(n)}(t)}^n - B_{G(t)}^n \rrn > z \right) \nonumber \\ &+ \bp\left(\sup_{t \in [0, \infty)}\lln G^{(n)(t)} - G(t) \rrn \lln B_1^n \rrn > z \right)~\forall z > 0.
\end{align}
From Lemma~\ref{lem:B_G^n-G} with $q=1$, there exist constants $C_4''$, $K_4''$ and $\la_4''$ such that for $z = (C_4'' \frac{\sqrt{\log n}}{n^{1/4}}+ x)$ we have:
\beq\label{eq:B-br_n(G)^n-G-2}
\bp \lp \sup_{t \in [0, \infty)} \lln B_{G^{(n)}(t)}^n - B_{G^(t)}^n \rrn > z \rp < K_4'' e^{-\la_4'' x^2 \sqrt{n}}.
\eeq
Recall notation $r_n(G)$ introduced in \eqref{eq:Gn-dkw}. Since $z \geq x$ we have
$$
\bp \lp \sup_{t \in [0, \infty)} \lln G^{(n)}(t) - G(t) \rrn |B_1^n| > z \rp = \bp \lp |B_1^n| > \dfrac{z}{r_n(G)} \rp \leq \bp \lp |B_1^n| > \dfrac{x}{r_n(G)} \rp.
$$
Since $r_n(G) = O(\frac{1}{\sqrt{n}})$ and $\bp (|B_1^n| > u) \leq 2 e^{-u^2/2}$, we now have:
\beq\label{eq:B-br_n(G)^n-G-3}
\bp \lp \sup_{t \in [0, \infty)} \lln G^{(n)}(t) - G(t) \rrn |B_1^n|  > z \rp \leq 2 e^{- \frac{x^2}{2 {r_n(G)}^2}} \leq 2 e^{-\la_4''' x^2 n},
\eeq
for some constant $\la_4'''$. Using \eqref{eq:B-br_n(G)^n-G-2} and \eqref{eq:B-br_n(G)^n-G-3} in \eqref{eq:B-br_n(G)^n-G-1} we obtain that there exist constants $C_6$, $K_6$ and $\la_6$ such that:
$$
\bp \lp \sup_{t \in [0, \infty)} \lln \tilde{B}_{G^{(n)}(t)}^{\textnormal{br},n} - \tilde{B}_{G(t)}^{\textnormal{br},n} \rrn > C_6 \dfrac{\sqrt{\log n}}{n^{1/4}} + x \rp < K_6 e^{-\la_6 x^2\sqrt{n}}.
$$
\end{proof}
}

We now arrive at our main result for this section, namely a strong embedding for the arrival process $A_n$. 
\begin{proposition}\label{prop:dropout-strong}
Let Assumption~\ref{assum:arriv-dropouts} or \ref{assum:arriv-dropouts-2} hold. Then there exists a Brownian motion $\hat{B}$, a Brownian bridge $B^{\textnormal{br},n}$ such that if $\hat{H}_n$ be defined as:
$$
\hat{H}_n(t) = 
\begin{cases}
\sqrt{n} p G(t) + p B_{G(t)}^{\textnormal{br},n} + \sqrt{p(1-p)} \hat{B}_{G(t)}, &\text{under Assum.~\ref{assum:arriv-dropouts}},\\
\sqrt{n} p \lp G(t) + r_n(G) \rp + p B_{G(t)}^{\textnormal{br},n} + \sqrt{p(1-p)} \hat{B}_{G(t)}, &\text{under Assum.~\ref{assum:arriv-dropouts-2}},
\end{cases}
$$
then there exists a version of $T_1, \ldots, T_n$, a version of $\zeta_1, \ldots, \zeta_n$, along with constants $C_7$, $K_7$ and $\la_7$ such that
$$
\bp \lp \sup_{t \in [0, \infty)} \lln \dfrac{A_n(t)}{\sqrt{n}} - \hat{H}_n(t) \rrn > C_7 \dfrac{\sqrt{\log n}}{n^{1/4}} + x \rp < 
 K_7 e^{-\la_7 x^2 \sqrt{n}}. 
$$
\end{proposition}
\begin{proof}

\noindent
\emph{Step 1: Assumption~\ref{assum:arriv-dropouts}}: From Proposition~\ref{prop:1}, there exists a Brownian motion $\hat{B}$ along with constants $\hat{C}_1$, $\hat{K}_1$ and $\hat{\la}_1$ such that for all $n \geq 1$ and $x > 0$ we have:
\beq\label{eq:arrival-a}
\bp \lp \sup_{t \in [0, \infty)} \lln \dfrac{1}{\sqrt{n}} \sum_{i=1}^{nG_n(t)} \dfrac{(\zeta_i - p)}{\sqrt{p(1-p)}} - \hat{B}_{G_n(t)}  \rrn > \hat{C}_1 \dfrac{\log n}{\sqrt{n}} + x \rp < \hat{K}_1 e^{-\hat{\la}_1 x \sqrt{n}}.
\eeq
From Proposition~\ref{prop:2}, there exists a Brownian bridge $B^{\textnormal{br},n}$ along with constants $\hat{C}_2$, $\hat{K}_2$ and $\hat{\la}_2$ such that for all $n \geq 1$ and $x > 0$ we have:
\beq\label{eq:arrival-b}
\bp \lp \sup_{t \in [0, \infty)} \lln \sqrt{n} \lp G_n(t) - G(t) \rp - B_{G(t)}^{\textnormal{br},n} \rrn > \hat{C}_2 \dfrac{\log n}{\sqrt{n}} + x \rp <\hat{K}_2 e^{-\hat{\la}_2 x \sqrt{n}}.
\eeq
Observe that the arrival process $A_n$ given by \eqref{eq:arriv-dropouts} may be decomposed as follows:\beq\label{eq:arriv-decom}
\dfrac{A_n(t)}{\sqrt{n}} = \sqrt{p(1-p)} A_{1,n}(t) + p A_{2,n}(t) + \sqrt{p(1-p)} A_{3,n}(t)  + \hat{H}_n(t),
\eeq
where 
\begin{align*}
A_{1,n}(t) &= \dfrac{1}{\sqrt{n}} \sum_{i=1}^{nG_n(t)} \dfrac{(\zeta_i - p)}{\sqrt{p(1-p)}} - \hat{B}_{G_n(t)}\\
A_{2,n}(t) &= \sqrt{n} \lp G_n(t) - G(t) \rp - B_{G(t)}^{\textnormal{br},n},
\end{align*}
and 
$$
A_{3,n}(t) = \hat{B}_{G_n(t)} - \hat{B}_{G(t)}.
$$
From Corollary~\ref{cor:1}, there exist constants $\hat{C}_3$, $\hat{K}_3$ and $\hat{\la}_3$ such that for all $n \geq 1$ and $x > 0$:
\beq\label{eq:arrival-c}
\bp \lp \sup_{t \in [0, \infty)} \lln A_{3,n}(t) \rrn > \hat{C}_3 \dfrac{\sqrt{\log n}}{n^{1/4}} + x \rp < \hat{K}_3 e^{-\hat{\la}_3 x^2 \sqrt{n}}.
\eeq
Using the bounds \eqref{eq:arrival-a}, \eqref{eq:arrival-b} and \eqref{eq:arrival-c} in the decomposition~\eqref{eq:arriv-decom}, we now obtain existence of constants $C_7$, $K_7$ and $\la_7$ such that for all $n \geq 1$ and $x > 0$:
$$
\bp \lp \sup_{t \in [0, \infty)} \lln \dfrac{A_n(t)}{\sqrt{n}} - \hat{H}_n(t) \rrn > C_7 \dfrac{\sqrt{\log n}}{n^{1/4}} + x \rp < K_7 e^{-\la_7 x^2 \sqrt{n}}.
$$

\noindent
\emph{Step 2: Assumption~\ref{assum:arriv-dropouts-2}}: From Proposition~\ref{prop:1}, there exists a Brownian motion $\hat{B}$ along with constants $\hat{C}_1$, $\hat{K}_1$ and $\hat{\la}_1$ such that for all $n \geq 1$ and $x > 0$ we have:
\beq\label{eq:arrival-2-a}
\bp \lp \sup_{t \in [0, \infty)} \lln \dfrac{1}{\sqrt{n}} \sum_{i=1}^{nG_n^{(n)}(t)} \dfrac{(\zeta_i - p)}{\sqrt{p(1-p)}} - \hat{B}_{G_n^{(n)}(t)}  \rrn > \hat{C}_1 \dfrac{\log n}{\sqrt{n}} + x \rp < \hat{K}_1 e^{-\hat{\la}_1 x \sqrt{n}}.
\eeq
From Corollary~\ref{cor:prop-2-cor}, there exists a Brownian bridge $B^{\textnormal{br},n}$ along with constants $\hat{C}_2$, $\hat{K}_2$ and $\hat{\la}_2$ such that for all $n \geq 1$ and $x > 0$ we have:
\beq\label{eq:arrival-2-b}
\bp \lp \sup_{t \in [0, \infty)} \lln \sqrt{n} \lp G_n^{(n)}(t) - G^{(n)}(t) \rp - B_{G^{(n)}(t)}^{\textnormal{br},n} \rrn > \hat{C}_2 \dfrac{\log n}{\sqrt{n}} + x \rp <\hat{K}_2 e^{-\hat{\la}_2 x \sqrt{n}}.
\eeq
Observe that the arrival process $A_n$ given by \eqref{eq:arriv-dropouts} may be decomposed as follows:
\begin{multline}\label{eq:arriv-decom-2}
\dfrac{A_n(t)}{\sqrt{n}} = \sqrt{p(1-p)} A_{1,n}(t) + p A_{2,n}(t) + \sqrt{p(1-p)} A_{3,n}(t) + p A_{4,n}(t) + pA_{5,n}(t)\\ + \lp \hat{H}_n(t)-\sqrt{n}pr_n(G) \rp,
\end{multline}
where 
\begin{align*}
A_{1,n}(t) &= \dfrac{1}{\sqrt{n}} \sum_{i=1}^{nG_n^{(n)}(t)} \dfrac{(\zeta_i - p)}{\sqrt{p(1-p)}} - \hat{B}_{G_n^{(n)}(t)},\\
A_{2,n}(t) &= \sqrt{n} \lp G_n^{(n)}(t) - G^{(n)}(t) \rp - B_{G^{(n)}(t)}^{\textnormal{br},n},\\
A_{3,n}(t) &= \lp \hat{B}_{G_n^{(n)}(t)} - \hat{B}_{G(t)} \rp,\\
A_{4,n}(t) &= B_{G^{(n)}(t)}^{\textnormal{br},n} - B_{G(t)}^{\textnormal{br},n},
\end{align*}
and 
$$
A_{5,n}(t) = \sqrt{n} \lp G^{(n)}(t) - G(t) \rp.
$$
From Corollary~\ref{cor:1} and Lemma~\ref{lem:B_G^n-G}, there exist constants $\hat{C}_3$, $\hat{K}_3$ and $\hat{\la}_3$ such that for all $n \geq 1$ and $x > 0$:
\beq\label{eq:arrival-2-c}
\bp \lp \sup_{t \in [0, \infty)} \lln A_{3,n}(t) \rrn > \hat{C}_3 \dfrac{\sqrt{\log n}}{n^{1/4}} + x \rp < \hat{K}_3 e^{-\hat{\la}_3 x^2 \sqrt{n}}.
\eeq
Next from Corollary~\ref{cor:B-br_n(G)^n-G}, there exist constants $\hat{C}_4$, $\hat{K}_4$ and $\hat{\la}_4$ such that for all $n \geq 1$ and $x > 0$ we have:
\beq\label{eq:arrival-2-d}
\bp \lp \sup_{t \in [0, \infty)} \lln A_{4,n}(t) \rrn > \hat{C}_4 \dfrac{\sqrt{\log n}}{n^{1/4}} + x  \rp < \hat{K}_4 e^{-\hat{\la}_4 x^2 \sqrt{n} \}}.
\eeq
Ultimately note that according to Assumption~\ref{assum:arriv-dropouts-2} we have: 
\beq\label{eq:arrival-2-e}
\sup_{t \in [0, \infty)} \lln A_{5,n}(t) \rrn  = \sup_{t \in [0, \infty)} \sqrt{n} \lp G^{(n)}(t) - G(t) \rp = \sqrt{n} r_n(G) < \infty
\eeq
Using the bounds \eqref{eq:arrival-2-a}, \eqref{eq:arrival-2-b}, \eqref{eq:arrival-2-c}, \eqref{eq:arrival-2-d} and \eqref{eq:arrival-2-e} in the decomposition \eqref{eq:arriv-decom-2}, we now obtain existence of constants $C_7$, $K_7$ and $\la_7$ such that for all $n \geq 1$ and $x > 0$:
$$
\bp \lp \sup_{t \in [0, \infty)} \lln \dfrac{A_n(t)}{\sqrt{n}} - \hat{H}_n(t) \rrn > C_7 \dfrac{\sqrt{\log n}}{n^{1/4}} + x \rp\\ < K_7 e^{-\la_7 x^2 \sqrt{n}}.
$$
\end{proof}

\section{A Strong Embedding for the Workload Process}\label{sec:w}
In this section we derive a strong embedding for the workload process, as well as the total remaining workload.

\begin{proposition}\label{prop:workload}
Let Assumptions~\ref{assum:arriv-dropouts} or \ref{assum:arriv-dropouts-2} and \ref{assum:1} hold. Then there exist Brownian motions $B$ and $\hat{B}$ and a Brownian bridge $B^{\textnormal{br}}$ such that if $\hat{R}_n$ be defined as:
$$
\hat{R}_n(t) = \si B_{p G(t)} + \mu \hat{H}_n(t),
$$
where $\hat{H}_n$ has been defined in Proposition~\ref{prop:dropout-strong},
then there exists a version of $T_1, \ldots, T_n$, a version of $V_1,\ldots, V_n$, a version of $\zeta_i, \ldots, \zeta_n$ along with constants $C_8$, $K_8$ and $\la_8$ such that for all $n \geq 1$ and $x > 0$:
$$
\bp \lp \sup_{t \in [0, \infty)} \lln \dfrac{W_n(t)}{\sqrt{n}} - \hat{R}_n(t) \rrn > C_8 \dfrac{\sqrt{\log n}}{n^{1/4}} + x \rp \leq K_8 e^{-\la_8 x^2 \sqrt{n}}.
$$
\end{proposition}
\begin{proof}
From Proposition~\ref{prop:1}, there exists a Brownian motion ${B}$ along with constants $\hat{C}_1$, $\hat{K}_1$ and $\hat{\la}_1$ such that for all $n \geq 1$ and $x > 0$ we have:
\beq\label{eq:workload-a}
\bp \lp \sup_{t \in [0, \infty)} \lln \dfrac{1}{\sqrt{n}} \sum_{i=1}^{A_n(t)} \dfrac{(V_i - \mu)}{\si} - {B}_{\frac{A_n(t)}{n}}  \rrn > \hat{C}_1 \dfrac{\log n}{\sqrt{n}} + x \rp < \hat{K}_1 e^{-\hat{\la}_1 x \sqrt{n}}.
\eeq
Observe that the workload process ${W_n}$ given by \eqref{eq:W_n} may be decomposed as follows:
\beq\label{eq:workload-decomp}
\dfrac{W_n(t)}{\sqrt{n}} = \si W_{1,n}(t) + \mu W_{2,n}(t) + \si W_{3,n}(t)  + \hat{R}_n(t),
\eeq
where
\begin{align*}
W_{1,n}(t) &= \dfrac{1}{\sqrt{n}} \sum_{i=1}^{A_n(t)} \dfrac{(V_i - \mu)}{\si} - B_{\frac{A_n(t)}{n}}\\
W_{2,n}(t) &= \dfrac{A_n(t)}{\sqrt{n}} - \hat{H}_n(t),
\end{align*} 
and
$$
W_{3,n}(t) =  B_{\frac{A_n(t)}{n}} - B_{pG(t)}.
$$
From Proposition~\ref{prop:dropout-strong}, there exist constants $\hat{C}_1$, $\hat{K}_1$ and $\hat{\la}_1$ such that for all $n \geq 1$ and $x > 0$:
\beq\label{eq:workload-b}
\bp \lp \sup_{t \in [0, \infty)} \lln W_{2,n}(t) \rrn > \hat{C}_2 \dfrac{\sqrt{\log n}}{n^{1/4}} + x \rp < \hat{K}_2 e^{-\hat{\la}_2 x^2 \sqrt{n}}.
\eeq
From Corollary~\ref{cor:1}, there exists constants $\hat{C}_3$, $\hat{K}_3$ and $\hat{\la}_3$ such that for all $n \geq 1$ and $x > 0$:
\beq\label{eq:workload-c}
\bp \lp \sup_{t \in [0, \infty)} \lln W_{3,n}(t) \rrn > \hat{C}_3 \dfrac{\sqrt{\log n}}{n^{1/4}} + x \rp < \hat{K}_3 e^{- \hat{\la}_3 x^2 \sqrt{n}}.
\eeq
Using the bounds \eqref{eq:workload-a}, \eqref{eq:workload-b} and \eqref{eq:workload-c} in the decomposition~\eqref{eq:workload-decomp}, we now obtain existence of constants $C_8$, $K_8$ and $\la_8$ such that for all $n \geq 1$ and $x > 0$:
$$
\bp \lp \sup_{t \in [0, \infty)} \lln \dfrac{W_n(t)}{\sqrt{n}} - \hat{R}_n(t) \rrn > C_8 \dfrac{\sqrt{\log n}}{n^{1/4}} + x \rp \leq K_8 e^{-\la_8 x^2 \sqrt{n}}.
$$
\end{proof}
The following result helps in extending strong embedding of processes to strong embedding of their Skorohod reflections defined in the sequel.
\begin{lemma}\label{lem:ana-result-1}
Let $f$ and $g$ be real-valued functions defined on $[0,\infty)$. Assume $f$ and $g$ satisfy the following property:
\beq\label{eq:sup-diff<del}
\sup_{t \in [0, \infty)} \lln f(t) - g(t) \rrn < \der.
\eeq
Then we have the following:
\beq\label{eq:sup-inf-diff<=del}
\sup_{t \in [0,\infty)} \lln \inf_{0\leq u \leq t} f(u) - \inf_{0\leq u \leq t} g(u)\rrn \leq \der.
\eeq
\end{lemma}
\begin{proof}
Observe that \eqref{eq:sup-inf-diff<=del} is immediate if the following holds for every $\ep>0$:
\beq\label{eq:sup-inf-diff<del+ep}
\sup_{t \in [0,\infty)} \lln \inf_{0\leq u \leq t} f(u) - \inf_{0\leq u \leq t} g(u)\rrn < \der + \ep.
\eeq
Let us prove \eqref{eq:sup-inf-diff<del+ep} through contradiction. First, assume the contrary, namely, there exists $t_0 \in [0, \infty)$ such that
$$
\lln \inf_{0\leq u \leq t_0} f(u) - \inf_{0\leq u \leq t_0} g(u) \rrn \geq \der + \ep.
$$
Consequently, assume without loss of generality:
\beq\label{eq:lem-1}
\inf_{0 \leq u \leq t_0} f(u) < \inf_{0 \leq u \leq t_0} g(u) -(\der + \ep).
\eeq
Observe there exist points $t_1^{\ep}, t_2^{\ep}$ such that 
\beq\label{eq:lem-2}
f(t_1^{\ep}) < \inf_{0 \leq u \leq t_0} f(u) + \ep, ~\text{and}~g(t_2^{\ep}) < \inf_{0 \leq u \leq t_0} g(u) + \ep.
\eeq
From \eqref{eq:sup-diff<del} and \eqref{eq:lem-2} we obtain:
\beq\label{eq:lem-3}
g(t_1^{\ep}) - \der < f(t_1^{\ep}) < \inf_{0 \leq u \leq t_0} f(u) + \ep.
\eeq
Finally, from \eqref{eq:lem-1} and \eqref{eq:lem-3} we obtain
$$
g(t_1^{\ep}) < \inf_{0 \leq u \leq t_0} g(u),
$$
which contradicts the definition of infimum and is not true. Hence our assumption \eqref{eq:lem-1} is wrong, and we must have
\beq\label{eq:lem-4}
\inf_{0 \leq u \leq t_0} f(u) \geq \inf_{0 \leq u \leq t_0} g(u) -(\der + \ep).
\eeq
Interchanging $f$ and $g$ in \eqref{eq:lem-4} allows us to conclude \eqref{eq:sup-inf-diff<del+ep}, and hence \eqref{eq:sup-inf-diff<=del} as well. 
\end{proof}
We now arrive at a strong embedding of the total remaining workload.
\begin{proposition}\label{prop:vw-sa}
Let $\phi$ be the reflection map functional given by:
$$
\phi(f)(t) = f(t) - \inf_{u \leq t} f(u).
$$
Then, under the same assumptions and notations as in Proposition~\ref{prop:workload} there exist constants $C_9$, $K_9$ and $\la_9$ such that 
$$
\bp \lp \sup_{t \in [0,\infty)} \lln \dfrac{1}{\sqrt{n}} \phi(W_n - c_n \cdot \textnormal{id})(t) - \phi(\hat{R}_n - \dfrac{c_n}{\sqrt{n}} \cdot \textnormal{id})(t) \rrn > C_9 \dfrac{\sqrt{\log n}}{{n}^{1/4}} + x \rp < K_9 e^{-\la_9 x^2 \sqrt{n}},
$$
where $\textnormal{id}:x \to x$ denotes the identity function, and $c_n$ is a positive constant denoting the server efficiency rate.
\end{proposition}
\begin{proof}
Observe from Proposition~\ref{prop:workload}, we have:
\beq\label{eq:prop-5-1}
1-K_8e^{-\la_8 x^2 \sqrt{n}} < \bp \lp \sup_{t \in [0,\infty)} \lln \dfrac{W_n(t)}{\sqrt{n}} - \hat{R}_n(t) \rrn < C_8 \dfrac{\sqrt{\log n}}{n^{1/4}} + x \rp.
\eeq
From Lemma~\ref{lem:ana-result-1} we have
\begin{align}\label{eq:prop-5-2}
\bp &\lp \sup_{t \in [0,\infty)} \lln \dfrac{W_n(t)}{\sqrt{n}} - \hat{R}_n(t) \rrn < C_8 \dfrac{\sqrt{\log n}}{n^{1/4}} + x \rp \nonumber \\
  &= \bp \lp \sup_{t \in [0, \infty)} \lln \dfrac{W_n(t) - c_n t}{\sqrt{n}} - \lp \hat{R}_n(t) - \dfrac{c_n t}{\sqrt{n}} \rp \rrn < C_8 \dfrac{\sqrt{\log n}}{n^{1/4}} + x \rp \nonumber\\ 
&\leq \bp \lp \sup_{t \in [0,\infty)} \lln \inf_{0\leq u\leq t} \dfrac{W_n(u) - c_n u}{\sqrt{n}} - \inf_{0 \leq u \leq t} \lp \hat{R}_n(u) - \dfrac{c_n u}{\sqrt{n}} \rp \rrn \leq C_8 \dfrac{\sqrt{\log n}}{n^{1/4}} + x \rp.
\end{align}
Combining equations \eqref{eq:prop-5-1} and \eqref{eq:prop-5-2}, and recalling the definition of $\phi$ we obtain:
$$
1-K_9e^{-\la_9 x^2 \sqrt{n}} < \bp \lp \sup_{t \in [0,\infty)} \lln \dfrac{1}{\sqrt{n}} \phi\lp W_n - c_n \cdot \textnormal{id} \rp (t) - \phi \lp \hat{R}_n - \dfrac{c_n}{\sqrt{n}} \cdot \textnormal{id} \rp (t) \rrn \leq C_9 \dfrac{\sqrt{\log n}}{n^{1/4}} + x \rp,
$$
for some constants $C_9$, $K_9$ and $\la_9$.
\end{proof}

\section{A Strong Embedding for the Queue Length Process}\label{sec:q}
In this section we obtain a strong approximation to the queue length process. Control of the truncated renewal process $M_n$ ({recall relation~\ref{eq:M_n}}) would lead to a strong approximation of the queue length.

\begin{lemma}\label{lem:Q-1}
Let $Z_n(t)$ be given by:
$$
Z_n(t) = c_n \lp \dfrac{t}{\mu} - \dfrac{M_n(t)}{c_n} \rp.
$$
Then there exist constants $C_{10}$, $K_{10}$ and $\la_{10}$ such that 
\beq\label{eq:Zn-bnd}
\bp \lp \sup_{t \in [0, \infty)} \lln \dfrac{ Z_n(t)}{\sqrt{n}} - \dfrac{\si}{\mu} B_{\frac{M_n(t)}{n}} \rrn > C_{10} \dfrac{\log n}{\sqrt{n}} + x \rp < K_{10} e^{-\la_{10} x \sqrt{n}}.
\eeq
\end{lemma}
\begin{proof}

 Notice $Z_n(t)$ can be further decomposed as:
\beq\label{eq:Z_n-decomp}
Z_n(t) = \sum_{i=1}^{M_n(t)} \dfrac{(V_i-\mu)}{\mu} - \dfrac{1}{\mu} \lp \sum_{i=1}^{M_n(t)} V_i - c_n t \rp.
\eeq
Henceforth, the two terms in \eqref{eq:Z_n-decomp} will be approximated. Utilizing the definition of $M_n(t)$, the second term in \eqref{eq:Z_n-decomp} can be bounded as follows:
$$
\dfrac{1}{\mu} \lln \sum_{i=1}^{M_n(t)} V_i - c_n t \rrn \leq \dfrac{V_{(M_n(t) +1)\wedge n}}{\mu}.
$$
Hence for any constant $C>0$ we obtain:
\begin{multline*}
\bp \lp \sup_{t \in [0, \infty)} \dfrac{1}{\sqrt{n} \mu} \lln \sum_{i=1}^{M_n(t)} V_i - c_n t \rrn > C \dfrac{\log n}{\sqrt{n}} + x  \rp\\ \leq \bp \lp \dfrac{V_i}{\mu} > C {\log n} + x \sqrt{n}, \text{ for all } i = 1, \ldots, n \rp
= {\lp \bp \lp V_1 > C \mu \log n + \mu x \sqrt{n} \rp \rp}^n.
\end{multline*}
Using Chebyshev's inequality we obtain:
\begin{multline}\label{eq:Zn-bnd-1}
\bp \lp \sup_{t \in [0, \infty)} \dfrac{1}{\sqrt{n} \mu} \lln \sum_{i=1}^{M_n(t)} V_i - c_n t \rrn > C \dfrac{\log n}{\sqrt{n}} + x  \rp \leq {\lp \dfrac{\bex\lc e^{\der \frac{V_1}{n}} \rp}{e^{C \der \mu \frac{\log n}{n} + \der \mu \frac{x}{\sqrt{n}}}} \rp}^n \\
{\dfrac{\bex[e^{\der V_1}]}{e^{C \der \mu \log n + \der \mu x \sqrt{n}}}}\leq K e^{- \la x \sqrt{n}},
\end{multline}
for some constants $K$ and $\la$, where the last step is obtained by choosing $\der$ sufficiently small.

In order to approximate the first term in \eqref{eq:Z_n-decomp}, observe that from Proposition~\ref{prop:1} there exist constants $C_1$, $K_1$ and $\la_1$ such that:
$$
\bp \lp \sup_{t \in [0, \infty)} \lln \dfrac{1}{\sqrt{n}} \sum_{i=1}^{M_n(t)} \lp V_i - \mu \rp - \si B_{\frac{M_n(t)}{n}} \rrn > C_1 \dfrac{\log n}{\sqrt{n}} + x \rp < K_1 e^{- \la_1 x \sqrt{n}},
$$
which implies
\beq\label{eq:Z_n-decomp-1}
\bp \lp \sup_{t \in [0, \infty)} \lln \dfrac{1}{\sqrt{n}} \sum_{i=1}^{M_n(t)} \dfrac{(V_i - \mu)}{\mu} - \dfrac{\si}{\mu} B_{\frac{M_n(t)}{n}} \rrn > \dfrac{C_1}{\mu} \dfrac{\log n}{\sqrt{n}} + x \rp < K_1 e^{-\la_1 \mu x \sqrt{n}}.
\eeq
Using \eqref{eq:Zn-bnd-1} and \eqref{eq:Z_n-decomp-1}, our desired inequality \eqref{eq:Zn-bnd} is obtained.
\end{proof}
We now approximate the Brownian motion evaluated at $M_n$ appearing in Lemma~\ref{lem:Q-1}
\begin{lemma}\label{lem:B_M_nt/n}
There exist constants $C_{11}$, $K_{11}$ and $\la_{11}$ such that
$$
\bp \lp \sup_{0 \leq t \leq L_n} \lln B_{M_n(t)/n} - B_{(\frac{c_n}{n}\frac{t}{\mu}) \wedge 1} \rrn > C_{11} \dfrac{\sqrt{\log \lp \frac{c_n}{n} \frac{L_n}{\mu} \vee n \rp}}{n^{1/4}} + x \rp < K_{11} e^{-\la_{11} x^2 \sqrt{n}}.
$$
\end{lemma}
\begin{proof}
It suffices to check the conditions in Proposition~\ref{prop:3}, which are satisfied by Proposition~\ref{lem:tBM-satisfy-2}. Combining equations \eqref{eq:str-app-ineq-1}, \eqref{eq:str-app-ineq-2} and \eqref{eq:str-app-ineq-3} we obtain:
$$
\bex\lc \sup_{s\in [0, L_n]} \lln \dfrac{M_n(t)}{n} - \dfrac{c_n}{n} \dfrac{t}{\mu} \wedge 1 \rrn \rc \leq \dfrac{\sup_{0 \leq k \leq n} \lln \tilde{S}_k - k \rrn}{n} + \dfrac{2}{n} + \dfrac{\bex|S_n -n\mu|}{n \mu}.
$$
Using the fact that $\bex(\sup_{0\leq k \leq n} |\tilde{S}_k - k|) \leq C \bex|\tilde{S}_n - n|$, we obtain:
$$
\bex\lc \sup_{s\in [0, L_n]} \lln \dfrac{M_n(t)}{n} - \dfrac{c_n}{n} \dfrac{t}{\mu} \wedge 1 \rrn \rc \leq \dfrac{C}{\sqrt{n}}.
$$
This yields our desired result.
\end{proof}

\begin{remark}\label{rem:c_n}
Notice from definition, $M_n(t)$ equals $n$ for all $t > \frac{S_n}{c_n}$. Hence, control of $B_{\frac{M_n(t)}{n}}$ for $t \in [0, \infty)$ reduces to a control of $B_{\frac{M_n(t)}{n}}$ for $t \in [0, \frac{S_n}{c_n}]$. However Lemma~\ref{lem:B_M_nt/n} leads a control of $B_{\frac{M_n(t)}{n}}$ over $t \in [0,L_n]$ for a predetermined and fixed sequence $L_n$. Hence we need $\frac{S_n}{c_n}$ to be in an interval $[0,L_n]$ with exponentially high probability (i.e. the complement event has exponentially decreasing probability). This is achieved in the following lemma.
\end{remark}

\begin{lemma}\label{lem:S_n/c_n}
Let $S_n = V_1 + \ldots + V_n$. Then for every $\eta>0$, there exists $\der > 0$ such that
$$
\bp \lp \dfrac{S_n}{c_n} \geq L_n \rp \leq \exp \lp -(\der c_n L_n - n\eta \rp.
$$
\end{lemma}

\begin{proof}
$$
\bp \lp \dfrac{S_n}{c_n} \geq L_n \rp \leq \dfrac{\bex e^{t S_n}}{e^{t c_n L_n}} \leq \dfrac{(\bex[e^{tV_1}] )^n}{e^{tc_n L_n}} \leq e^{-(\der c_n L_n - n \eta)},
$$
for some $\der$ small enough such that $\bex e^{\der V_1} < e^{\eta}$.
\end{proof}

From Lemma~\ref{lem:B_M_nt/n} and Remark~\ref{rem:c_n} we obtain the following result, which controls $B_{\frac{M_n(t)}{n}}$ for all $t$ positive.
\begin{lemma}\label{lem:B_M_nt/n-2}
$$
\bp \lp \sup_{t \in [0, \infty)} \lln B_{M_n(t)/n} - B_{(\frac{c_n}{n}\frac{t}{\mu}) \wedge 1} \rrn > C_{11} \dfrac{\sqrt{\log \lp \frac{c_n L_n}{n \mu}\vee n \rp}}{n^{1/4}} + x \Bigg\vert \dfrac{S_n}{c_n} \leq L_n \rp < K_{11} e^{-\la_{11} x^2 \sqrt{n}}.
$$
\end{lemma}
Finally we arrive at the main result for this section, namely a strong embedding for the queue length $Q_n$.
\begin{proposition}\label{prop:queue-length}
Let Assumptions~\ref{assum:arriv-dropouts} or \ref{assum:arriv-dropouts-2} and \ref{assum:1} hold. Let $\tilde{E}_n$ be given by:
\beq\label{eq:Q-tildeE_n}
\tilde{E}_n(t) =
\begin{cases}
pG(t) - \dfrac{c_n}{n}\dfrac{t}{\mu},~&\text{under Assum.~\ref{assum:arriv-dropouts}},\\
p \lp G(t) + r_n(G) \rp - \dfrac{c_n}{n}\dfrac{t}{\mu},~&\text{under Assum.~\ref{assum:arriv-dropouts-2}}.
\end{cases} 
\eeq
Recall $\hat{H}_n$ defined in Proposition~\ref{prop:dropout-strong} and define the process $\hat{Y}_n$ as follows:
$$
\hat{Y}_n(t) = \lp \hat{H}_n(t) - \dfrac{c_n t}{\sqrt{n} \mu} \rp + \dfrac{\si}{\mu} B_{( \frac{c_n t}{n \mu} + \inf_{s \leq t} \tilde{E}_n(s))}.
$$
Then there exists a version of $T_1, \ldots, T_n$, a version of $V_1,\ldots, V_n$, a version of $\zeta_i, \ldots, \zeta_n$ along with constants $C_{12}, K_{12}$, $\la_{12}$ and $\zeta$ independent of $n$, such that for all $n \geq 1$ and $x>0$:
$$
\bp \lp \sup_{t \in [0, \infty)} \lln \dfrac{Q_n(t)}{\sqrt{n}} - \phi(\hat{Y}_n)(t) \rrn > C_{12} \dfrac{\sqrt{\log n}}{{n^{1/4}}} + x \rp < K_{12} e^{-\la_{12} \sqrt{n} x^2 \wedge x} + e^{-n \zeta},
$$
if $c_n = O(n^m)$ for some $m > 0$ and $\liminf_n c_n > 0$, else:
$$
\bp \lp \sup_{t \in [0, \infty)} \lln \dfrac{Q_n(t)}{\sqrt{n}} - \phi(\hat{Y}_n)(t) \rrn > C_{12} \dfrac{\sqrt{\log c_n}}{{n^{1/4}}} + x \rp < K_{12} e^{-\la_{12} \sqrt{n} x^2 \wedge x}  + e^{-n \zeta},
$$
\end{proposition}
\begin{proof}
Observe that from \eqref{eq:q-length} the diffusion-scaled queue length $\frac{Q_n}{\sqrt{n}}$ can be further decomposed as:
$$
\dfrac{Q_n(t)}{\sqrt{n}} = Y_n(t) + \dfrac{c_n}{\sqrt{n}\mu} I_n(t),
$$
where the idle time process $I_n$ has been defined in \eqref{eq:idle}, $Y_n$ is given by:
\beq\label{eq:Q-Y_n}
Y_n(t) :=  \lp \dfrac{A_n(t)}{\sqrt{n}} - \hat{H}_n(t) \rp + \dfrac{c_n}{\sqrt{n}} \lp \dfrac{D_n(t)}{\mu} - \dfrac{M_n(D_n(t))}{c_n}\rp + \lp \hat{H}_n(t) - \dfrac{c_n}{\sqrt{n}} \dfrac{t}{\mu} \rp,
\eeq
and $\hat{H}_n$ has been defined in Proposition~\ref{prop:dropout-strong}.
By the Skorohod reflection theorem we have that 
\beq\label{eq:i-time-skor}
\dfrac{c_n}{\sqrt{n} \mu} I_n(t) = -  \inf_{s \leq t} Y_n(s),
\eeq
and the busy time process is given by:
$$
D_n(t) = t+ \dfrac{\sqrt{n} \mu}{c_n}\inf_{s \leq t} Y_n(s). 
$$
The diffusion-scaled queue length process is now given by the Skorohod reflection of $Y_n$:
\beq\label{eq:Q-refl}
\dfrac{Q_n(t)}{\sqrt{n}} = \phi(Y_n)(t).
\eeq
Thus a strong embedding of $\frac{Q_n}{\sqrt{n}}$ would follow from a strong embedding of $Y_n$. Notice that we already have a strong embedding of the arrival process courtesy Proposition~\ref{prop:dropout-strong}, namely there exist constants $C_7$, $K_7$ and $\la_7$ such that for all $n \geq 1$ and $x > 0$:
\beq\label{eq:Q-A_n}
\bp \lp \sup_{t \in [0,\infty)} \lln \dfrac{A_n(t)}{\sqrt{n}} - \hat{H}_n(t) \rrn > C_7 \dfrac{\log n}{\sqrt{n}} + x \rp < K_7 e^{-\la_7 x^2 \sqrt{n}}.
\eeq
Thus to complete the strong embedding of $Y_n$, we need to approximate $\tilde{Z}_n$ given by:
$$
\tilde{Z}_n(t) := c_n \lp \dfrac{D_n(t)}{\mu} - \dfrac{M_n(D_n(t))}{c_n} \rp.
$$
Observe that the busy time process $D_n(t)$ is non-decreasing and takes values in $[0, \frac{S_n}{c_n}]$ where $S_n = \sum_{i=1}^n V_i$. Consequently  from Lemma~\ref{lem:Q-1}, there exist constants $\hat{C}_1$, $\hat{K}_1$ and $\hat{\la}_1$ such that for all $n \geq 1$ and $x > 0$:
\beq\label{eq:Q-tildeZ}
\bp \lp \sup_{t \in [0, \infty)} \lln \dfrac{\tilde{Z}_n(t)}{\sqrt{n}} - \dfrac{\si}{\mu} B_{\frac{M_n(D_n(t))}{n}} \rrn > \hat{C}_1 \dfrac{\log n}{\sqrt{n}} + x \rp < \hat{K}_1 e^{-\hat{\la}_1 x \sqrt{n}}. 
\eeq
Let
\beq\label{eq:X-tilde}
\tilde{Y}_n(t) = \lp \hat{H}_n(t) - \dfrac{c_n t}{\sqrt{n} \mu} \rp + \dfrac{\si}{\mu} {B_{\frac{M_n(D_n(t))}{n}}}.
\eeq
Using \eqref{eq:Q-A_n} and \eqref{eq:Q-tildeZ} in \eqref{eq:Q-Y_n} we thus have constants $\hat{C}_2$, $\hat{K}_2$ and $\la_2$ such that for all $n \geq 1$ and $x > 0$:
\beq\label{eq:Y_n-tildeY_n}
\bp \lp \sup_{t \in [0, \infty)} \lln {Y_n(t)} - \tilde{Y}_n(t) \rrn > \hat{C}_2 \dfrac{\log n}{\sqrt{n}} + x \rp < \hat{K}_2 e^{-\hat{\la}_2  x \sqrt{n}}.
\eeq
From the converse of Lemma~\ref{lem:ana-result-1}, there exist constants $\hat{C}_3$, $\hat{K}_3$ and $\hat{\la}_3$ such that for all $n \geq 1$ and $x > 0$:
$$
\bp \lp \sup_{t \in [0, \infty)} \lln \inf_{s \leq t} Y_n(s) - \inf_{s \leq t} \tilde{Y}_n(s) \rrn > \hat{C}_3 \dfrac{\log n}{\sqrt{n}} + x \rp < \hat{K}_3 e^{-\hat{\la}_3 x \sqrt{n}}.
$$ 
Recalling the expression \eqref{eq:idle} and using \eqref{eq:i-time-skor} we have for every $n \geq 1$ and $x > 0$:
$$
\bp \lp \sup_{t \in [0, \infty)} \lln \dfrac{c_n}{\sqrt{n} \mu} (t -D_n(t)) + \inf_{s \leq t} \tilde{Y}_n(s) \rrn > \hat{C}_3 \dfrac{\log n}{\sqrt{n}} + x \rp < \hat{K}_3 e^{-\hat{\la}_3 x \sqrt{n}}.
$$
Consequently for every $n \geq 1$ and $x > 0$:
\beq\label{eq:busy-time-1}
\bp \lp \sup_{t \in [0, \infty)} \lln \dfrac{c_n}{{n} \mu} D_n(t) - \lp \dfrac{c_n}{n} \dfrac{t}{\mu} + \dfrac{1}{\sqrt{n}} \inf_{s \leq t} \tilde{Y}_n(s) \rp \rrn > \hat{C}_3 \dfrac{\log n}{{n}} + x \rp < \hat{K}_3 e^{-\hat{\la}_3 x {n}}.
\eeq
Recall $\tilde{E}_n$ given by \eqref{eq:Q-tildeE_n}. From the expression of $\tilde{Y}_n$ in \eqref{eq:X-tilde} and recalling $\hat{H}_n$ from Proposition~\ref{prop:dropout-strong} we have:
\begin{align}\label{eq:X-tilde_DKW}
&\bp \lp \sup_{t \in [0, \infty)} \lln \dfrac{\tilde{Y}_n(t)}{\sqrt{n}} - \tilde{E}_n(t) \rrn > \ep \rp \nonumber \\
&\leq \bp \lp \sup_{t \in [0, \infty)} \lln  \dfrac{p B_{G(t)}^{\textnormal{br},n}}{\sqrt{n}} \rrn + \sup_{t \in [0, \infty)} \lln \dfrac{\sqrt{p(1-p)} \hat{B}_{G(t)}}{\sqrt{n}} \rrn +\sup_{t \in [0, \infty)}  \lln \dfrac{\si}{\mu} \dfrac{B_{M_n(D_n(t))/n}}{\sqrt{n}} \rrn > \ep \rp \nonumber \\
&\leq \bp \lp \sup_{t \in [0, \infty)} \lln B_{G(t)}^{\textnormal{br},n} \rrn > \dfrac{\sqrt{n} \ep}{3p} \rp + \bp \lp \sup_{t \in [0, \infty)} \lln \hat{B}_{G(t)} \rrn > \dfrac{\sqrt{n} \ep}{3 \sqrt{p(1-p)}} \rp\\ \nonumber & \qquad +\bp \lp \sup_{t \in [0, \infty)} \lln B_{M_n(D_n(t))/n} \rrn > \dfrac{\sqrt{n}\mu }{\si} \dfrac{\ep}{3} \rp. 
\end{align}
Observe that both $G(t)$ and $\frac{M_n(D_n(t))}{n}$ are less than $1$. Using the tail probability for the supremum of the standard Brownian bridge on $[0,1]$ we have:
\beq\label{eq:Q-B-1}
 \bp \lp \sup_{t \in [0, \infty)} \lln B_{G(t)}^{\textnormal{br},n} \rrn > \dfrac{\sqrt{n} \ep}{3p} \rp \leq 2 \exp \lp \dfrac{2 n \ep^2}{9p^2} \rp.
\eeq
Using the tail probability for the supremum of the Brownian motion on $[0,1]$ we have
\beq\label{eq:Q-B-2}
 \bp \lp \sup_{t \in [0, \infty)} \lln \hat{B}_{G(t)} \rrn > \dfrac{\sqrt{n} \ep}{3 \sqrt{p(1-p)}} \rp \leq 4  \int_{\sqrt{n} \ep /3\sqrt{p(1-p)}}^{\infty} \dfrac{e^{-s^2/2}}{\sqrt{2\pi}},
\eeq
and
\beq\label{eq:Q-B-3}
\bp \lp \sup_{t \in [0, \infty)} \lln B_{M_n(D_n(t))/n} \rrn > \dfrac{\sqrt{n}\mu }{\si} \dfrac{\ep}{2} \rp
\leq 4 \int_{\sqrt{n} \ep \mu/2\si}^{\infty} \dfrac{e^{-s^2/2}}{\sqrt{2\pi}}.
\eeq
Using \eqref{eq:Q-B-1}, \eqref{eq:Q-B-2} and \eqref{eq:Q-B-3} in \eqref{eq:X-tilde_DKW}, we have that there exist constants $k_1$ and $k_2$ such that:
$$
\bp \lp \sup_{t \in [0, \infty)} \lln \dfrac{\tilde{Y}_n(t)}{\sqrt{n}} -\tilde{E}_n(t) \rrn > \ep \rp \leq k_1 \exp(-k_2 n \ep^2).
$$
From the converse of Lemma~\ref{lem:ana-result-1}, there exist constants $\hat{k}_1$ and $\hat{k}_2$ such that:
\beq\label{eq:tilde-X-fluid}
\bp \lp \sup_{t \in [0, \infty)} \lln \inf_{s \leq t}\dfrac{\tilde{Y}_n(s)}{\sqrt{n}} - \inf_{s \leq t} \tilde{E}_n(s) \rrn > \ep \rp \leq \hat{k}_1 \exp(-\hat{k}_2 n \ep^2).
\eeq
Now, using \eqref{eq:busy-time-1} and \eqref{eq:tilde-X-fluid} we have:
\begin{multline}\label{eq:cnDn/n-approx-1}
\bp \lp \sup_{t \in [0, \infty)} \lln \dfrac{c_n}{n} \dfrac{D_n(t)}{\mu} - \lp \dfrac{c_n t}{n \mu} + \inf_{s \leq t} \tilde{E}_n(s) \rp \rrn > \hat{C}_3 \dfrac{\log n}{n} + 2 \ep \rp\\
\leq \bp \lp \sup_{t \in [0, \infty)} \lln \dfrac{c_n}{n} \dfrac{D_n(t)}{\mu} - \lp \dfrac{c_n}{n} \dfrac{t}{\mu} + \dfrac{1}{\sqrt{n}} \inf_{s \leq t} \tilde{Y}_n(s) \rp \rrn > \hat{C}_3 \dfrac{\log n}{n} + \ep \rp \qquad \\ + \bp \lp \sup_{t \in [0, \infty)} \lln \dfrac{1}{\sqrt{n}} \inf_{s \leq t} \tilde{Y}_n(s) - \inf_{s \leq t}\tilde{E}_n(s) \rrn > \ep \rp \leq \hat{K}_3 e^{-\hat{\la}_3 n \ep} + \hat{k}_1 \exp (-\hat{k}_2 n \ep^2). 
\end{multline}
Since $|x \wedge 1 - y \wedge 1| \leq |x-y|$ we have from \eqref{eq:cnDn/n-approx-1} constants $\hat{k}_3$, $\hat{k}_4$ and $\hat{k}_5$ such that:
\beq\label{eq:cnDn/n-approx-2}
\bp \lp \sup_{t \in [0, \infty)} \lln \lp \dfrac{c_n}{n} \dfrac{D_n(t)}{\mu}\rp \wedge 1 - \lp \dfrac{c_n t}{n \mu} + \inf_{s \leq t} \tilde{E}_n(s) \rp \wedge 1 \rrn > \hat{C}_3 \dfrac{\log n}{n} + \ep \rp \leq \hat{k}_3 e^{-\hat{k}_4 n \ep^2 \wedge \hat{k}_5 n \ep}. 
\eeq
From Proposition~\ref{lem:tBM-satisfy-2} we have constants $\hat{k}_6$ and $\hat{k}_7$ such that:
\beq\label{eq:MnDn-approx}
\bp \lp \sup_{t \in [0, \infty)} \lln \dfrac{M_n(D_n(t))}{n} - \lp \dfrac{c_n}{n} \dfrac{D_n(t)}{\mu}  \rp \wedge 1 \rrn >  \dfrac{2}{n} + \ep ~\Bigg\vert~  \dfrac{S_n}{c_n} \leq L_n  \rp \leq \hat{k}_6 e^{- \hat{k}_7 n \ep^2}.
\eeq
Combining \eqref{eq:cnDn/n-approx-2} and \eqref{eq:MnDn-approx} we have constants $\tilde{k}_0$, $\tilde{k}_1$, $\tilde{k}_2$ and $\tilde{k}_3$ such that:
$$
\bp \lp \sup_{t \in [0, \infty)} \lln \dfrac{M_n(D_n(t))}{n} - \lp \dfrac{c_n t}{n \mu} + \inf_{s \leq t}\tilde{E}_n(s) \rp \wedge 1 \rrn > \tilde{k}_0 \dfrac{\log n}{n} + \ep  ~\Bigg\vert~  \dfrac{S_n}{c_n} \leq L_n  \rp
 \leq \hat{k}_1 e^{-\hat{k}_2 n \ep^2 \wedge \hat{k}_3 n \ep}.
$$
Now, from Proposition~\ref{prop:3} we obtain constants $\hat{C}_4$, $\hat{K}_4$ and $\hat{\la}_4$ such that:
\beq\label{eq:B-SDt-approx}
\bp \lp \sup_{t \in [0, \infty)} \lln B_{\frac{M_n(D_n(t))}{n}} - B_{( \frac{c_n t}{n \mu} +  \inf_{s \leq t} \tilde{E}_n(s))} \rrn > \hat{C}_4 \dfrac{\sqrt{\log (\frac{c_n L_n}{n \mu} \vee n)}}{n^{1/4}} + x \Bigg\vert \dfrac{S_n}{c_n} \leq L_n \rp \leq \hat{K}_4 e^{-\hat{\la}_4 x^2\sqrt{n}}.
\eeq
Let $\hat{Y}_n$ be given by
$$
\hat{Y}_n(t) = \lp \hat{H}_n(t) - \dfrac{c_n t}{\sqrt{n} \mu} \rp + \dfrac{\si}{\mu} B_{( \frac{c_n t}{n \mu} + \inf_{s \leq t} \tilde{E}_n(s))}.
$$ 
We now obtain from \eqref{eq:X-tilde} and \eqref{eq:B-SDt-approx} existence of constants $\hat{C}_5$, $\hat{K}_5$ and $\hat{\la}_5$ such that :
$$
\bp \lp \sup_{t \in [0, \infty)} \lln \tilde{Y}_n(t) - \hat{Y}_n(t) \rrn > \hat{C}_5 \dfrac{\sqrt{\log(\frac{c_n L_n}{n \mu} \vee n)}}{n^{1/4}} + x \Bigg\vert \dfrac{S_n}{c_n} < L_n \rp \leq \hat{K}_5 e^{-\hat{\la}_5 x^2 \sqrt{n}}.
$$
Hence from \eqref{eq:Y_n-tildeY_n} we have constants $\hat{C}_6$, $\hat{K}_6$ and $\hat{\la}_6$ such that
$$
\bp \lp \sup_{t \in [0, \infty)} \lln Y_n(t) - \hat{Y}_n(t) \rrn > \hat{C}_6 \dfrac{\sqrt{\log(\frac{c_n L_n}{n \mu} \vee n)}}{n^{1/4}} + x \Bigg\vert \dfrac{S_n}{c_n} \leq L_n \rp \leq \hat{K}_6  e^{-\hat{\la}_6 \sqrt{n} x \wedge x^2}.
$$
Recalling \eqref{eq:Q-refl}, we now obtain that there exist constants $\hat{C}_7$, $\hat{K}_7$ and $\hat{\la}_7$ such that
$$
\bp \lp \sup_{t \in [0, \infty)} \lln \dfrac{Q_n(t)}{\sqrt{n}} - \phi(\hat{Y}_n)(t) \rrn > \hat{C}_7 \dfrac{\sqrt{\log(\frac{c_n L_n}{n \mu} \vee n)}}{n^{1/4}} + x \Bigg\vert \dfrac{S_n}{c_n} < L_n \rp \leq  \hat{K}_7 e^{-\hat{\la}_7 \sqrt{n} x \wedge x^2}$$
Observe that for any two sets $A$ and $B$, we have that $P(A) \leq P(A|B) + P(B^{c})$. Thus using Lemma~\ref{lem:S_n/c_n} we have:
$$
\bp \lp \sup_{t \in [0, \infty)} \lln \dfrac{Q_n(t)}{\sqrt{n}} - \phi(\hat{Y}_n)(t) \rrn > \hat{C}_7 \dfrac{\sqrt{\log(\frac{c_n L_n}{n \mu} \vee n)}}{n^{1/4}} + x \rp \leq  \hat{K}_7 e^{-\hat{\la}_7 \sqrt{n} x \wedge x^2}
 + e^{-(\der c_n L_n -n \eta)}.
$$
Finally choosing $L_n = n$ yields the desired result.
\end{proof}

\section{Commentary and Conclusions}\label{sec:conclusions}
From a philosophy of science perspective, one can consider the bulk of the nonstationary queueing model literature as {\it phenomenological} (\cite{FrHa2006}) in nature -- that is, accurately reflecting empirical evidence, but not necessarily first principles. For instance, as noted in~Remark 2.7 (see~\cite{Wh2018} as well) a widely used nonstationary traffic model uses a composition construction, where a cumulative intensity function that captures the time-varying effects is posited. However, these models are not necessarily a first principles explanation of how customers choose to arrive at a service system. This distinction between phenomenological and mechanistic modeling is not crucial from a performance analysis/prediction perspective, but it can be important from a system design or optimization and control perspective. For instance, in~\cite{ArAtHo2017} the authors study the design of an optimal appointment schedule to a nonstationary single server queueing system. In this instance, the composition constructed time inhomogeneous queueing models~\cite{Wh2018} are not appropriate and the authors used a transitory queueing model. On the other hand, stationary queueing models\footnote{we define this as the content in either \cite{gross} or~\cite{kleinrock}}, on the other hand, are also { mechanistic} models~\cite{FrHa2006} in nature whereby the arrival and service models are directly related to the behavior and choices of individual customers. The \rsg model provides a mechanistic description of queueing behavior and can be used in optimization, control and game theoretic models of queueing behavior with much flexibility.

The \rsg model generalizes the \dgi model that has been studied in the literature. However, computing performance metrics for the discrete event \rsg queueing model is quite difficult, owing to the complicated time-dependencies in the model. The strong embeddings (and FSATs) proved in this paper provide error bounds from tractable diffusion approximations. Furthermore, these results can be specialized to yield prior diffusion limits obtained via weak convergence in~\cite{HoJaWa2016a,HoJaWa2016b,BeHoLe2019}. We anticipate that our FSAT results to be immensely useful for optimization and control problems involving queueing systems. 

There are several avenues for further exploration. First, our most general conditions on the traffic model in Assumption~\ref{assum:arriv-dropouts-2} allows
 the arrival epoch distribution to depend on the population size -- allowing for the possibility that an increase in the population will change customer behavior since there is an increase in demand for services. We currently assume that the ``dropout'' probability is stationary. A more general model would allow for time-dependent/non-stationary dropout probabilities. It seems possible to extend the current FSETs and FSATs to this setting. More complicated is establishing analogous results for a multi-server queue. In~\cite{HoJaWa2016a}, diffusion limits were established for a fixed multi-server queue in the large population asymptotic limit, relying on the fact that in the large sample limit the regulator process is identical to the single server case. In our current paper, however, the FSETs (which are for finite $n$) are much harder to prove, since we can no longer use the asymptotic simplification. This issue is compounded when the servers are not identical, and we will investigate these results in future papers. A further avenue for investigation is how to prove FSETs in a scaling regime that is analogous to the many-server heavy-traffic (MSHT) scaling. In this case, we anticipate that the diffusion approximation should be some type of a non-stationary Halfin-Whitt diffusion process, but we have not been able to prove the FSET, and it appears we might require some new mathematical innovations to achieve this result.


 \section*{Appendix}
\begin{theorem}
Let $F$ be a distribution function with mean $0$ and variance $1$. In addition, suppose the moment generating function corresponding to $F$, $R(t) = \E (e^{tX})$, $X \sim F$, exists in a neighborhood of $0$. Then, given a Brownian motion $B$, and using it, one can construct a sequence of random variables $X_1, X_2, \ldots$ which are independent and identically distributed to $F$. Furthermore, the partial sums of $X_i$'s are strongly coupled to the Brownian motion $B$ in the following sense. For every $n \in \N$ and $x > 0$:
\beq
\bp \lp \sup_{1\leq k\leq n} \lln \sum_{i=1}^n X_i - B_n \rrn > C \log n + x \rp < K e^{-\la x},
\eeq
where $C$, $K$ and $\la$ are positive constants depending only on $F$.
\end{theorem}

\begin{proof}[Proof ideas and construction:]
Let us first give a sketch of the construction that yields the random variables $\mathbf{X}:=\{X_i, i\in \N\}$. Instead of fixing the Brownian motion it would suffice to  fix an infinite sequence of independent standard normals $\mathbf{Y}:=\{ Y_i, i \in \N \}$.  Let $G$ be the distribution function for a generic element from $\mathbf{Y}$ and recall $F$ is the distribution function for a generic element from $\mathbf{X}$.

Observe that $F(X)$ and $G(Y)$ are both $U[0,1]$ distributed random variables, and a natural suggestion to generate $X$ given $Y$ (or vice versa) is to equate them:
\beq\label{eq:quant-trans}
F(X) = G(Y), ~\text{i.e.}~X=F^{-1}(G(Y)),
\eeq
assuming $F^{-1}$ is properly defined. Equation~\eqref{eq:quant-trans} is the quantile transform which forms the backbone of the construction, details of which are to follow.

The above technique turns out to be useful especially if one is interested in bounding $|X-Y|$ as well (which is our aim too), e.g., when $F$ is the distribution of sum of $n$ iid mean $0$ random variables with finite moment generating function, and $G$ is the distribution of appropriately normalized normal random variable, then
$$
|X-Y| \leq C_1 \dfrac{X^2}{n} + C_2 ~~~\text{if}~~|X| < \ep n.
$$
Note the error of approximation does not increase with the number of components in $X$.

If someone wants to approximate $S_n = X_1+\cdots + X_n$ by $T_n = Y_1 + \cdots + Y_n$ (we are interested in $Y_i$'s distributed as normals), such that the components are independent, then we may divide $S_n$ into blocks, in particular if:
$$
S_n = X_1 + \cdots + \lp X_{n_i} + \cdots + X_{n_{i+1}} \rp + \cdots + X_n,
$$
we may approximate the block $X_{n_i} + \cdots +  X_{n_{i+1}}$ using $Y_{n_i} + \cdots Y_{n_{i+1}}$ whereby this approximation is first achieved by quantile transformation of  $Y_{n_i} + \cdots + Y_{n_{i+1}}$ and then finding the individual $X_i$'s by conditioning on the sum $X_{n_i} + \cdots + X_{n_{i+1}}$ (details to follow). This is the key idea in the KMT construction.

Suppose $(X_{m+1} + \cdots + X_{m+2n})$ has already been obtained as quantile transform of $(Y_{m+1}+ \cdots + Y_{m+2n})$. For normally distributed $Y_i$'s we have $[(Y_{m+1}+\cdots+Y_{m+n}) + (Y_{m+n+1} + \cdots + Y_{m+2n})]$ independent of $[(Y_{m+1} + \cdots + Y_{m+n}) - (Y_{m+n+1} + \cdots + Y_{m+2n})]$. Thus one may expect:
$$
U_{2n} := \lc \lp X_{m+1} + X_{m+n} \rp + \lp X_{m+n+1} + \cdots + X_{m+2n} \rp \rc
$$
to be approximately independent of 
$$
\tilde{U}_{2n} := \lc \lp X_{m+1}+ \cdots + X_{m+n} \rp - \lp X_{m+n+1} + \cdots + X_{m+2n} \rp \rc.
$$
Also let 
$$
V_{2n} := \lc \lp Y_{m+1} + \cdots + Y_{m+n} \rp +\lp Y_{m+n+1} + \cdots + Y_{m+2n} \rp \rc,
$$
and
$$
\tilde{V}_{2n} = \lc \lp Y_{m+1} + \cdots + Y_{m+n} \rp  - \lp Y_{m+n+1} + \cdots + Y_{m+2n} \rp \rc.
$$

As a consequence, one can obtain two blocks out of $U_{2n}$, having already obtained $U_{2n}$ from quantile transform of $V_{2n}$: 
\begin{enumerate}
\item[(i)] Consider $F(x|y) = P (\tilde{U}_{2n} < x | U_{2n} = y)$. 
\item[(ii)] Transform $G(V_{2n}/\sqrt{2n})$ which is $U[0,1]$ distributed by the application of the quantile transform, i.e., 
$$
\tilde{U}_{2n}| U_{2n} = F^{-1} \lp \cdot | U_{2n} \rp \lp G \lp \dfrac{V_{2n}}{\sqrt{2n}} \rp \rp.
$$
\item[(iii)] Now, let 
$$
X_{m+1} + \cdots + X_{m+n} = (U_{2n} + \tilde{U}_{2n})/2,
$$ 
and 
$$
X_{m+n+1} + \cdots + X_{m+2n} = (U_{2n} - \tilde{U}_{2n})/2.
$$
Observe that $(U_{2n} + \tilde{U}_{2n})/2$ and $(U_{2n} - \tilde{U}_{2n})/2$ are independent.
\end{enumerate}

In the construction these steps are applied iteratively, until the individual random variables get constructed. To this end, the following structure/blocking is considered in the beginning: the set of standard normals $\mathbf{Y}$ is broken down into: $\{Y_1, Y_2\} \cup_{j=1}^{\infty} \{ Y_{2^j+1}, \ldots, Y_{2^{j+1}} \}$. From each of these blocks we construct $\{X_1, X_2\} \cup_{j=1}^{\infty}\{ X_{2^j +1} , \ldots, X_{2^{j+1}} \}$ and then apply the three steps outlined above repeatedly until we arrive at individual random elements. This is possible because the block sizes are powers of $2$. 

The random variables $X_i$'s thus generated are clearly independent of each other. To see this, observe that the blocks $\{ Y_{2^{j}+1}, \ldots, Y_{2^{j+1}} \}$ are independent and hence the blocks $\{ X_{2^j + 1}, \ldots, X_{2^{j+1}} \}$ are independent as well. Each division of the block also generate independent random variables. Thus the newly constructed $\mathbf{X}$ is composed of independent entries. Furthermore each of them is distributed with distributed with distribution function $F$. 

The only remaining argument is to show that the $X_i$'s thus generated are close to the $Y_i$'s. For some $m \in \N$, let $j = \max \{ i: 2^i | m \}$. Then $m=(2k+1)2^j$ for some $k \in \N$. One can write $S_m$ as follows:
$$
2S_m = S_{(2k+2)2^j}  + S_{(2k)2^j} + (S_{(2k+1)2^j} - S_{(2k)2^j}) - (S_{(2k+2)2^j} - S_{(2k+1)2^j}),
$$
or written according to the notations in \cite{KMT}:
$$
2S_m = S_{(2k+2)2^j} + S_{(2k)2^j} + \tilde{U}_{j+1,k}.
$$
Applying this recursion repeatedly for $2^n < m \leq 2^{n+1}$ one arrives at:
$$
S_m = \tilde{S}_m + \sum_{i=j+1}^n c(i) \tilde{U}_{i, k(i)}, 
$$
where $\tilde{S}_m$ is the linear interpolation between $S_{2^n}$ and $S_{2^{n+1}}$ given by:
$$
\tilde{S}_m = \dfrac{2^{n+1}-m}{2^n} S_{2^n} + \dfrac{m-2^n}{2^n} S_{2^{n+1}},
$$
and $c(i)$ and indices $k(i)$ depend on $m$, $0 \leq c(i) \leq 1$, $k(i) = \lceil \frac{m}{2^{i}} \rceil$.

A similar representation is possible for $T_n = Y_1 + \cdots + Y_n$, $n \in \N$. One obtains:
$$
\lln (S_m - T_m) - (\tilde{S}_m - \tilde{T}_m) \rrn \leq \sum_{i=j+1}^n \lln \tilde{U}_{i, k(i)} - \tilde{V}_{i, k(i)} \rrn.
$$
Lemma~1 in \cite{KMT} gives bounds on these differences. These are used in a string of inequalities to bound the probability:
\beq\label{eq:main-prob}
P \lp \sup_{1 \leq m \leq 2^N} |S_m - T_m| > x\rp.
\eeq
Existence of moment generating function is required in order to apply Chernoff inequality to obtain the best possible bound. Note that non-existence of moment generating function will not provide an exponentially decreasing upper bound to \eqref{eq:main-prob}.
\end{proof}

\begin{theorem}
There exists a probability space with independent $U[0,1]$ random variables  $U_1, U_2, \ldots$ and a sequence of Brownian bridges $B_1^{\textnormal{br}}, B_2^{\textnormal{br}}, \ldots$ such that for all $n \geq 1$ and $x \in \R$:
\beq
\bp \lp \sup_{s \in [0,1]} \sqrt{n} \lln \al_n(s) - B_n^{br}(s) \rrn > C \log n + x \rp < K e^{-\la x},
\eeq
where the empirical process $\al_n$ is given by
$$
\al_n (s) = \sqrt{n} (F_n(s) - s), 
$$
and 
$$
F_n(s) = \dfrac{1}{n} \sum_{i=1}^n \mathbf{1}_{\{U_i \leq s\}}.
$$
\end{theorem}

\begin{proof}[Proof ideas and construction:]
The following has been borrowed extensively from \cite{mattner} and the reader is encouraged to look at Sections 25.1-25.3 for a clearer exposition.  Here as well, we would construct the random variables $U_1, \ldots, U_n$ (for a fixed $n$) from the Brownian bridge $B_n^\textnormal{br}$. In fact, a similar dyadic scheme is used to construct the required variables. 

The first key idea in the construction is the fact that every real number in $[0,1]$ has a binary representation (terminating for rationals, while non-terminating for irrationals). 

\begin{enumerate}
\item[(i)] Let $\mathscr{Z} = \{ Z \} \cup_{j=1}^{\infty} \{ Z_i: i \in {\{ 0,1 \}}^j \}$ be an indexed set of independent standard normal random variables.

\item[(ii)] Let $H_m$ be the inverse distribution function for $\textnormal{Bin}(m, \frac{1}{2})$. Define $H_0 = 0$.

\item[(iii)] Let 
$$
N_0 = H_n(\Phi(Z))~~ \text{and}~~ N_1 = n - H_n(\Phi(Z)).
$$
Next for each $i \in {\{0,1\}}^j$, $j \geq 1$:
$$
N_{i,0} = H_{N_i} (\Phi(Z_i))~~\text{and}~~N_{i,1}=N_i - H_{N_i}(\Phi(Z_i)). 
$$
\end{enumerate}
Observe that $n = N_0+N_1$, and for each $i \in {\{ 0,1 \}}^j$, $j \geq 1$ we have $N_i = N_{i,0} + N_{i,1}$. This implies for each $j \geq 1$, $N_i, i \in {\{0,1\}}^j$ is $\textnormal{Multinomial}(n, \frac{1}{2^j} \cdots \frac{1}{2^j})$. 

We have obtained $\{N_i , i \in {\{ 0,1 \}}^j, j\geq 1\}$ from standard normals. Using these $N_i$'s we will now construct a uniform empirical distribution function $G_n$ as follows:
For $j \geq 1$ and $k = 1, \ldots, 2^j$ let
$$
G_n^j\lp\dfrac{k}{2^j}\rp = \dfrac{1}{n} \sum_{i \in A_{j,k}} N_i,
$$
where 
$$
A_{j,k} = \lcl i \in {\{ 0,1 \}}^j: \sum_{s=1}^j \dfrac{i_s}{2^j} \leq \dfrac{k}{2^j} ~\text {with}~ i=(i_1, \ldots,i_j) \rcl.
$$
With increasing $j$, the $G_n^j$'s become finer and finer. Thus by taking limits (as rationals are dense in [0,1]) we can construct a uniform empirical distribution function defined on the entire interval $[0,1]$:
$$
G_n = \lim_{j \to \infty} G_n^j.
$$

Inverting $G_n$ one obtains the order statistics:
$$
U_{(1)} \leq \cdots \leq U_{(n)}
$$
of $n$ independent $U[0,1]$ random variables. A random permutation of these order statistics provides a sample of $n$ independent $U[0,1]$ random variables.

Our aim is to construct $\al_n(s) = \sqrt{n}(G_n(s) - s)$, such that it is close to the Brownian bridge, $B^\textnormal{br}$. This means that the set of standard normals we considered previously should be obtained from the Brownian bridge in a \emph{nice manner}. The following is the \emph{nice manner} (here we have denoted the $n^{\text{th}}$ Brownian bridge as $B$):
\begin{enumerate}
\item[(i)] $Z=2B\lp \frac{1}{2} \rp$.
\item[(ii)] For $i \in {\{ 0,1 \}}^j$, $j \geq 1$:
$$
Z_i = 2^{j/2} \lp 2B\lp \dfrac{2k+1}{2^{j+1}} \rp - B\lp \dfrac{2k}{2^{j+1}} \rp -B\lp \dfrac{2k+2}{2^{j+1}} \rp \rp,
$$
where $k$ is obtained from $\frac{k}{2^j} = \sum_{s=1}^j \frac{i_s}{2^s}$ $(i = (i_1, \ldots, i_j))$.

\item[(iii)] It can be checked that these $Z_i$'s are standard normals and independent of each other.
\end{enumerate} 

Next one needs to argue why this construction of the uniform random variables produces an empirical process which is close to the Brownian bridge in the sense of \eqref{eq:strong-approx-2}. Some sketch of the proof can be found in \cite{Mason-notes}.
\end{proof}
Notice the following remark already alluded to in the main text.
\begin{remark}
The above construction of uniform random variables is for a fixed $n$. Having obtained $\{U_1, \ldots, U_n\}$, one is unable to obtain another $U_{n+1}$ such that the new set $\{U_1, \ldots, U_{n+1}\}$ satisfies \eqref{eq:strong-approx-2} with the same Brownian bridge. Instead we have to redo our construction. This necessitates the need for a different Brownian bridge $B^{\textnormal{br},n}$ for every $n$. However the following statement is true:

\emph{
There exists a Brownian bridge $B$ such that for each $n$, there exists $n$ independent uniforms $U_1, \ldots, U_n$ whose empirical process $\al_{n}$ satisfies:
$$
\bp \lp \sup_{s \in [0,1]} \sqrt{n} \lln \al_n(s) - B(s) \rrn > C \log n + x \rp < K e^{-\la x}.
$$
}
\end{remark}

\bibliographystyle{}
\bibliography{references}

\end{document}